\documentclass[twoside]{amsart}

\numberwithin{equation}{section}
\usepackage{amsmath,amssymb,amsfonts,amsthm,latexsym}
\usepackage{graphicx,color,enumerate}
\usepackage[margin=1in]{geometry}

\newtheorem{theorem}{Theorem}[section]

\newtheorem{corollary}[theorem]{Corollary}

\theoremstyle{definition}

\newtheorem{example}[theorem]{Example}
\theoremstyle{remark}
\newtheorem{remark}[theorem]{Remark}

\newcommand{\C}{\mathbb{C}}
\newcommand{\Sp}{\mathbb{S}}

\newcommand{\R}{\mathbb{R}}

\newcommand{\Z}{\mathbb{Z}}
\newcommand{\N}{\mathbb{N}}

\newcommand{\U}{\operatorname{U}}
\newcommand{\SU}{\operatorname{SU}}
\newcommand{\SL}{\operatorname{SL}}
\newcommand{\SO}{\operatorname{SO}}
\newcommand{\OO}{\operatorname{O}}

\newcommand{\co}{\colon\thinspace}

\newcommand{\bs}{\boldsymbol}

\newcommand{\Hilb}{\operatorname{Hilb}}
\newcommand{\on}{\text{\rm on}}
\newcommand{\off}{\text{\rm off}}

\newcommand{\HOM}{\operatorname{Hom}}
\newcommand{\Alt}{\bs{\operatorname{Alt}}}

\begin{document}

\title[Multigraded Hilbert series of rank $1$ Lie groups]
{Multigraded Hilbert series of invariants, covariants, and symplectic quotients for some rank 1 Lie groups}

\author[A.~Barringer]{Austin Barringer}
\address{
1090 Vermont Ave NW Suite 900   \\
Washington, DC, 20005, USA}
\email{austin.barringer@yahoo.com}

\author[H.-C.~Herbig]{Hans-Christian Herbig}
\address{
Departamento de Matem\'{a}tica Aplicada,
Universidade Federal do Rio de Janeiro,
Av. Athos da Silveira Ramos 149,
Centro de Tecnologia - Bloco C,
CEP: 21941-909 - Rio de Janeiro, Brazil}
\email{herbighc@gmail.com}

\author[D.~Herden]{Daniel Herden}
\address{
Department of Mathematics,
Baylor University,
Sid Richardson Building,
1410 S.4th Street,
Waco, TX 76706, USA}
\email{daniel\_herden@baylor.edu}

\author[S.~Khalid]{Saad Khalid}
\address{The Ohio State University,
Department of Physics,
191 W Woodruff Ave,
Columbus, OH 43210, USA}
\email{khalid.42@buckeyemail.osu.edu}

\author[C.~Seaton]{Christopher Seaton}
\address{
Department of Mathematics and Computer Science,
Rhodes College,
2000 N. Parkway,
Memphis, TN 38112, USA}
\email{seatonc@rhodes.edu}

\author[L.~Walker]{Lawton Walker}
\address{
Georgia Institute of Technology,
North Ave NW,
Atlanta, GA 30332
}
\email{jlwalker708@gmail.com}

\thanks{
A.B., S.K., and L.W. were supported by Rhodes College Summer Research Fellowships.
H.-C.H. was supported by CNPq through the \emph{Plataforma Integrada Carlos Chagas}.
C.S. was supported by the E.C.~Ellett Professorship in Mathematics.
}


\begin{abstract}
We compute univariate and multigraded Hilbert series of invariants and covariants of representations of the circle and orthogonal group $\OO_2$. The multigradings considered include the maximal grading associated to the decomposition of the representation into irreducibles as well as the bigrading associated to a cotangent-lifted representation, or equivalently, the bigrading associated to the holomorphic and antiholomorphic parts of the real invariants and covariants. This bigrading induces a bigrading on the algebra of on-shell invariants of the symplectic quotient, and the corresponding Hilbert series are computed as well. We also compute the first few Laurent coefficients of the univariate Hilbert series, give sample calculations of the multigraded Laurent coefficients, and give an example to illustrate the extension of these techniques to the semidirect product of the circle by other finite groups. We describe an algorithm to compute each of the associated Hilbert series.
\end{abstract}

\maketitle
\tableofcontents

% xxxxxxxxxxxxxxxxxxxxxxxxxxxxxxxxxxxxxxxxxxxxxxxxxxxxxxxxxxxxxxxxxxxxxxxxx
% xxxxxxxxxxxxxxxxxxxxxxxxxxxxxxxxxxxxxxxxxxxxxxxxxxxxxxxxxxxxxxxxxxxxxxxxx
% xxxxxxxxxxxxxxxxxxxxxxxxxxxxxxxxxxxxxxxxxxxxxxxxxxxxxxxxxxxxxxxxxxxxxxxxx

\section{Introduction}
\label{sec:Intro}

Let $R = \bigoplus_{d=0}^\infty R_d$ be an $\N$-graded commutative algebra over a field $\Bbbk$ where $\N = \{0,1,2,\ldots\}$. The Hilbert series of $R$
is the generating function of the dimensions of the $R_d$,
\[
    \Hilb_R(t)  =   \sum\limits_{d=0}^\infty t^d \dim_{\Bbbk} R_d.
\]
If $R$ is finitely generated, then $\Hilb_R(t)$ is the power series of a rational function with radius of
convergence at least $1$; see \cite[Section~1.4]{DerksenKemperBook}.
More generally, if $R = \bigoplus_{\bs{d}\in\N^n} R_{\bs{d}}$ is $\N^n$-graded and $M = \bigoplus_{\bs{d}\in\N^n} M_{\bs{d}}$
is an $\N^n$-graded $R$-module, the \emph{multigraded Hilbert series of $M$ associated to the $\N^n$-grading} is given by
\[
    \Hilb_M(\bs{t})  =   \sum\limits_{\bs{d}\in\N^n} \bs{t}^{\bs{d}} \dim_{\Bbbk} M_{\bs{d}},
\]
where $\bs{t} = (t_1,\ldots,t_n)$ and for $\bs{d} = (d_1,\ldots,d_n)$, $\bs{t}^{\bs{d}} = t_1^{d_1}\cdots t_n^{d_n}$;
see \cite[Chapter~1, Section~2]{StanleyCombComAlg}. Once again, if $R$ is finitely generated and $M$ is finitely generated over $R$,
then $\Hilb_M(\bs{t})$ is the power series of a rational function in the $t_i$; see \cite[Chapter~1, Theorem~2.3]{StanleyCombComAlg}.

If $R = \Bbbk[V]^G$ is the algebra of
invariant polynomials of a finite-dimensional representation $V$ of a reductive group $G$ with the usual $\N$-grading, then
$\Hilb_R(t)$ can be computed using Molien's formula when $G$ is finite or the Molien-Weyl formula in general; see
\cite[Theorem~3.4.2, Section~4.6.1]{DerksenKemperBook}, \cite[Theorem~2.2.1]{SturmfelsBook}, or Theorem~\ref{thrm:MolienWeyl} below.
Suppose the $G$-module $V$ decomposes as $V = V_1\oplus\cdots\oplus V_n$ where we do not necessarily
assume that the $V_i$ are irreducible. Consider the $\N^n$-grading of $\Bbbk[V]$ and $\Bbbk[V]^G$ given by assigning
to a monomial $f_1(\bs{v}_1) f_2(\bs{v}_2)\cdots f_n(\bs{v}_n)$ with $f_i(\bs{v}_i)\in\Bbbk[V_i]$ the degree
$(\deg f_1, \ldots, \deg f_n) \in \N^n$. We refer to this grading as the \emph{grading associated to the decomposition
$V_1\oplus\cdots\oplus V_n$ of $V$}. Similarly, the multigraded Hilbert series $\Hilb_{\Bbbk[V]^G}(t_1,\ldots,t_n)$
is the \emph{multigraded Hilbert series associated to the decomposition}.
More generally, if $W$ is another finite-dimensional representation of $G$, then the $\Bbbk[V]^G$-module of covariants
$\HOM(V,W)^G = (\Bbbk[V]\otimes W)^G$ is similarly graded by assigning to the element $f(\bs{v})\otimes w$, where $f\in\Bbbk[V]$
is homogeneous with respect to the grading induced by the decomposition of $V$ and $w\in W$, the degree of $f(\bs{v})$.
As $G$ is reductive so that $\Bbbk[V]^G$ is a finitely-generated algebra and
$\HOM(V,W)^G$ is a finitely-generated $\Bbbk[V]^G$-module by \cite[Theorem~4.2.10]{DerksenKemperBook}, \cite[Theorem~3.24]{PopovVinberg},
the multigraded Hilbert series of invariants or covariants is again the power series of a rational function.
If the decomposition of $V$ is into irreducibles, we refer to the
resulting grading as the \emph{maximal grading} of $V$. When $G = \Sp^1$, then because a polynomial is
$\Sp^1$-invariant if and only if each of its monomial terms is invariant, the maximally graded Hilbert series
determines the algebra of invariants or covariants; see \cite[Chapter~1, Section~3]{StanleyCombComAlg} and
Example~\ref{ex:S1Max1} below.

An important case of the above multigrading is the \emph{cotangent lift} $V = V_1 \oplus V_1^\ast$ of the representation $V_1$;
we refer to the corresponding grading as the \emph{cotangent bigrading}. If $\Bbbk = \C$, then this grading corresponds to
the decomposition of the underlying real representation into holomorphic and antiholomorphic components.
Specifically, we can express elements of $V$ with coordinates
$(z_1,\ldots,z_n,w_1,\ldots,w_n)$ where the $z_i$ are coordinates for $V_1$, the $w_i$ are dual coordinates for $V_1^\ast$,
the $\N^2$-degree of each $z_i$ is $(1,0)$, and the degree of $w_i$ is $(0,1)$.
The real polynomial invariants of the real representation underlying $V_1$ correspond to the invariants of $V$ with real coefficients
where $w_i = \overline{z_i}$; see \cite[Proposition~(5.8)(I)]{GWSliftingHomotopies}. The cotangent bigraded Hilbert series
was defined and studied in \cite{Forger}.

One may also consider the case of the \emph{on-shell invariants} of the \emph{linear symplectic quotient} associated
to a unitary representation $V$ of a compact Lie group $G$; see \cite[Section~2]{HerbigSeaton}.
Letting $J\co V\to\mathfrak{g}^\ast\simeq\R^\ell$ denote the \emph{moment map} associated to the $G$-representation, the on-shell invariants are
defined to be $\R[V]^G/I_J^G$ where $I_J$ is the ideal of $\R[V]$ generated by the components of $J$ and $I_J^G = I_J\cap\R[V]^G$.
Because the components of the moment map have degree $(1,1)$ with respect to the cotangent bigrading and are not homogeneous with
respect to further decompositions of $V$, the only reasonable gradings for the on-shell invariants are those induced by the standard
$\N$-grading or the cotangent bigrading on $V$. In the case that $I_J$ is a \emph{real} ideal, the on-shell invariants
coincide with the real regular functions on the corresponding symplectic quotient; see
\cite{ArmsGotayJennings}, \cite[Section~2]{HerbigIyengarPflaum}, \cite[Section~2.1]{FarHerSea}, \cite[Section~4]{HerbigSchwarz}, or
\cite[Section~2.2]{HerbigSchwarzSeaton2}. Note that we will sometimes refer to elements of $\Bbbk[V]^G$ as \emph{off-shell invariants}
to avoid potential confusion with the on-shell invariants.

This paper continues a program of computing Hilbert series which we began in \cite{HerbigSeaton,CowieHerbigSeatonHerden} for univariate Hilbert
series of on- and off-shell invariants of circle representations, \cite{CayresPintoHerbigHerdenSeaton,HerbigHerdenSeatonSU2} for
on- and off-shell invariants associated to $\SL_2$- and $\SU_2$-modules where multigraded Hilbert series played a minor role,
and \cite{HerbigHerdenSeatonT2} for the univariate Hilbert series of on-shell invariants of $2$-torus modules. See also
\cite{Hilbert,Brion,LittelmannProcesi,BroerSysBinForm,BroerCM,BroerCovariants,BroerNewMethod}, the more recent work
\cite{BedratyukPoincareCovariants,BedratyukWeitzenbock1,BedratyukIlashCovariants,IlashPoincareNLinForm,IlashPoincareNQuadForm,BedratyukXin},
and in particular the recent work involving multigraded Hilbert series
\cite{BedratyukBivarPoincare,BedratyukSL2MultigradPoincare,BedratyukBedratyukMultivarPoincare}.
See also \cite{StanleyHilbFunGradAlg,PopovBook,FarHerSea,HerbigSeaton2,HerbigHerdenSeaton,CapeHerbigSeaton,HerbigSchwarzSeaton2,HerbigLawlerSeaton,HerbigHerdenSeatonLaurent}
for applications of Hilbert series and Laurent coefficient computations to identifying properties of and distinguishing between
the corresponding singularities,
\cite{Hanany,HananyMekareeya,BaoHananyOpenProblems}
for applications in gauge theory, and
\cite{CollinsSzekelyhidi,CollinsSzekelyhidi2} for applications to K-stability of Sasakian manifolds.

The present paper focuses on the multigraded Hilbert series, which in particular avoids the complications introduced by degeneracies; see
\cite{HerbigSeaton,CowieHerbigSeatonHerden}. We consider the case of the circle $\Sp^1$ in Section~\ref{sec:S1}
as well as the non-connected orthogonal group $\OO_2$ in Section~\ref{sec:O2}. In each case,
we compute the maximally graded Hilbert series of the off-shell invariants and covariants and the bigraded Hilbert series
of the on-shell invariants. We also compute the maximally graded Hilbert series of invariants of the group $\Sp^1\rtimes\Z/4\Z$ in
Section~\ref{sec:O2OtherSemidirect} to indicate the extension of these techniques to other non-connected groups.
Specializing to the univariate Hilbert series, we compute the first two nontrivial coefficients of the Laurent expansion of the Hilbert series
of circle covariants in Section~\ref{subsec:S1Laurent}, generalizing the results of \cite{HerbigSeaton,CowieHerbigSeatonHerden}, as well as the
Laurent coefficients of covariants and on-shell invariants for $\OO_2$ in Section~\ref{subsec:O2Laurent}.
Though the Laurent coefficients in the multigraded case are ambiguous and less motivated, we give some sample computations of
bigraded Laurent coefficients in the case of the on-shell invariants of a circle representation in Section~\ref{subsec:S1LaurentMultigrad}.
In Section~\ref{sec:Algorithms}, we describe an algorithm to compute the Hilbert series in one case which can easily be extended to the other
cases computed in this paper.

% xxxxxxxxxxxxxxxxxxxxxxxxxxxxxxxxxxxxxxxxxxxxxxxxxxxxxxxxxxxxxxxxxxxxxxxxx
% xxxxxxxxxxxxxxxxxxxxxxxxxxxxxxxxxxxxxxxxxxxxxxxxxxxxxxxxxxxxxxxxxxxxxxxxx
% xxxxxxxxxxxxxxxxxxxxxxxxxxxxxxxxxxxxxxxxxxxxxxxxxxxxxxxxxxxxxxxxxxxxxxxxx

\section*{Acknowledgements}

H.-C.H. and C.S. would like to thank Baylor University, D.H. and C.S. would like to thank
the Instituto de Matem\'{a}tica Pura e Aplicada (IMPA), and H.-C.H. and D.H. would like to thank Rhodes College
for hospitality during work contained here.
This paper developed from A.B., S.K., and L.W.'s senior seminar projects in
the Rhodes College Department of Mathematics and Computer Science, and the
authors gratefully acknowledge the support of the department and college for these
activities. A.B., S.K., and L.W. express appreciation to the Rhodes College Summer Research Fellowship program,
H.-C.H. to CNPq through the \emph{Plataforma Integrada Carlos Chagas}, and
C.S. to the E.C.~Ellett Professorship in Mathematics and the Rhodes College sabbatical program,
for financial support.

% xxxxxxxxxxxxxxxxxxxxxxxxxxxxxxxxxxxxxxxxxxxxxxxxxxxxxxxxxxxxxxxxxxxxxxxxx
% xxxxxxxxxxxxxxxxxxxxxxxxxxxxxxxxxxxxxxxxxxxxxxxxxxxxxxxxxxxxxxxxxxxxxxxxx
% xxxxxxxxxxxxxxxxxxxxxxxxxxxxxxxxxxxxxxxxxxxxxxxxxxxxxxxxxxxxxxxxxxxxxxxxx

\section{Background}
\label{sec:Back}

In this section, we collect some background information which we will need throughout this paper. We refer the reader to
\cite{DerksenKemperBook,PopovVinberg}, \cite[Section~2]{HerbigSeaton}, and the explanations in the Introduction for more details.

We begin with the following, which will be our primary tool for computing Hilbert series.
For univariate Hilbert series, see \cite[Theorem~3.4.2]{DerksenKemperBook} or \cite[Theorem~2.2.1]{SturmfelsBook}
for the case of invariants of a finite group,
\cite[Remark~3.4.1]{DerksenKemperBook} for covariants of a finite group, and
\cite[Section~4.6.1]{DerksenKemperBook} for the case of a general reductive group. See
\cite[Equation~(13)]{StanleyInvarFinGrp} and \cite[Section~IV]{Forger} for extensions to the multigraded case;
the general result stated below is proven similarly.

\begin{theorem}[Molien-Weyl Theorem]
\label{thrm:MolienWeyl}
Let $G$ be a reductive group over $\C$ and $K$ a maximally compact subgroup. Let $V = V_1\oplus V_2\oplus\cdots\oplus V_n$ and $W$ be
finite-dimensional rational representations of $G$ over $\C$. Then the multigraded Hilbert series of the module of covariants $\HOM(V,W)^G$
associated to the decomposition $V = V_1\oplus V_2\oplus\cdots\oplus V_n$ is given by
\[
    \Hilb_{\HOM(V,W)^G}(\bs{t})
    =   \int\limits_{z\in K} \frac{\chi_W(z^{-1}) \, d\mu}{\prod\limits_{i=1}^n \det_{V_i} (1 - t_i z_i)},
\]
where $\bs{t} = (t_1,\ldots,t_n)$, $\mu$ is a normalized Haar measure on $K$, $\chi_W$ is the character of the representation $W$,
$\det_{V_i}$ denotes the determinant on $\operatorname{End}(V_i)$, and $z_i$ denotes the restriction of the
action of $z\in G$ to $V_i$.
\end{theorem}

Of course, the invariants $\C[V]^G$ correspond to the case when $W$ is the trivial $1$-dimensional representation.
Because the module of covariants is additive, i.e., $\HOM(V,W_1 \oplus W_2)^G = \HOM(V,W_1)^G \oplus \HOM(V,W_2)^G$, we will often
assume that $W$ is irreducible with no loss of generality.

Theorem~\ref{thrm:MolienWeyl} can be applied to the covariants of a compact Lie group $K$ by
letting $G$ denote the complexification of $K$, and to real invariants through the isomorphism $\R[V]^K\otimes\C\simeq \C[V]^K = \C[V]^G$
given in \cite[Proposition~(5.8)(I)]{GWSliftingHomotopies}. Note that $V$ and $W$ will always admit Hermitian scalar products with respect
to which the representations of $K$ are unitary. Because of these observations, we will state our results for invariants and
covariants of unitary representations of compact Lie groups, though the results extend directly to invariants and covariants of rational
representations of the complexifications of these groups.

The Laurent coefficients at $t = 1$ of the univariate Hilbert series considered here are most easily described using Schur
polynomials and variations thereof, so let us briefly recall these definitions and fix notation. Suppose
$\bs{x} = (x_1,\ldots,x_n)$ is a set of indeterminates. A \emph{partition} $\lambda = (\lambda_1,\ldots,\lambda_n)\in\Z^n$ of length $n$ is a set of integers $\lambda_1\geq\lambda_2\geq\cdots\geq\lambda_n\geq 0$. The \emph{alternant}
$\Alt_\lambda(\bs{x})$ associated to $\lambda$ in $\bs{x}$ is the determinant
$\Alt_\lambda(\bs{x}) = \lvert x_j^{\lambda_i}\rvert$; it is an alternating polynomial in the $x_i$. Letting
$\delta_n = (n-1,n-2,\ldots,1,0)$ and $\bs{V}(\bs{x}) = \Alt_{\delta_n}(\bs{x})$ denote the Vandermonde determinant,
any alternant is divisible by $\bs{V}(\bs{x})$, and the quotient is a symmetric polynomial.
The \emph{Schur polynomial associated to $\lambda$} is defined by
\begin{equation}
\label{eq:DefSchur}
    \bs{s}_\lambda(\bs{x})  =   \frac{\Alt_{\lambda+\delta_n}(\bs{x})}{\bs{V}(\bs{x})}.
\end{equation}
For more details and for a combinatorial description of $\bs{s}_\lambda(\bs{x})$, see
\cite[Section~I.3]{MacdonaldSymFuncs} and \cite[Sections~4.4--6]{SaganSymGrp}.

If $\bs{x} = (x_1,\ldots,x_k)$ and $\bs{y} = (y_1,\ldots,y_m)$ are two sets of indeterminates, $n = k + m$, and
$u\leq n-2$ is an integer, then the \emph{partial Laurent-Schur polynomial} $\bs{sp}_u(\bs{x},\bs{y})$ is defined by
\begin{equation}
\label{eq:DefPartialSchur}
    \bs{sp}_u(\bs{x},\bs{y})
    =   \frac{1}{\bs{V}(\bs{x})\bs{V}(\bs{y})}
        \begin{vmatrix}
            x_1^u       &   \cdots  &   x_k^u       &   0           &   \cdots  &   0
            \\
            x_1^{n-2}   &   \cdots  &   x_k^{n-2}   &   y_1^{n-2}   &   \cdots  &   y_m^{n-2}
            \\
            x_1^{n-3}   &   \cdots  &   x_k^{n-3}   &   y_1^{n-3}   &   \cdots  &   y_m^{n-3}
            \\
            \vdots      &           &   \vdots      &   \vdots      &           &   \vdots
            \\
            x_1         &   \cdots  &   x_k         &   y_1         &   \cdots  &   y_m
            \\
            1           &   \cdots  &   1           &   1           &   \cdots  &   1
        \end{vmatrix}.
\end{equation}
The partial Laurent-Schur polynomial $\bs{sp}_u(\bs{x},\bs{y})$ is a polynomial if $u\geq 0$ and a Laurent polynomial
if $u < 0$; it is symmetric separately in the $\bs{x}$ and $\bs{y}$. See \cite[Section~5]{CowieHerbigSeatonHerden}
for more details, an expression of $\bs{sp}_u(\bs{x},\bs{y})$ in terms of Schur polynomials in $\bs{x}$ and $\bs{y}$,
and a combinatorial description.

% xxxxxxxxxxxxxxxxxxxxxxxxxxxxxxxxxxxxxxxxxxxxxxxxxxxxxxxxxxxxxxxxxxxxxxxxx
% xxxxxxxxxxxxxxxxxxxxxxxxxxxxxxxxxxxxxxxxxxxxxxxxxxxxxxxxxxxxxxxxxxxxxxxxx
% xxxxxxxxxxxxxxxxxxxxxxxxxxxxxxxxxxxxxxxxxxxxxxxxxxxxxxxxxxxxxxxxxxxxxxxxx

\section{Invariants and covariants of $\Sp^1$}
\label{sec:S1}

Let us begin by briefly recalling the irreducible representations of $\Sp^1$; as explained in Section~\ref{sec:Back},
the results of this section apply as well to the complexification $\C^\times$ of $\Sp^1$.
Because $\Sp^1$ is abelian, its irreducible unitary representations are all $1$-dimensional. For each $a\in\Z$,
there is a unique irreducible representation $\epsilon_a\co\Sp^1\to\U_1$ given by $\epsilon_a(z) = z^a$ acting as
multiplication. Hence, a finite-dimensional unitary representation of $\Sp^1$ can be described by a \emph{weight vector}
$\bs{a} = (a_1,\ldots,a_n)$, indicating the representation $V = V_{\bs{a}}= \bigoplus_{i=1}^n\epsilon_{a_i}$.
We will often use the notation $\alpha_i = \lvert a_i\rvert$ to indicate the absolute values of the weights.

% xxxxxxxxxxxxxxxxxxxxxxxxxxxxxxxxxxxxxxxxxxxxxxxxxxxxxxxxxxxxxxxxxxxxxxxxx

\subsection{The maximally graded Hilbert series}
\label{subsec:S1Max}

In this section, we compute the maximally graded Hilbert series of the covariants of unitary $\Sp^1$-modules.

\begin{remark}
\label{rem:S1Faithful}
Assume the representation $V_{\bs{a}}$ is not faithful, i.e., $g:= \gcd(\alpha_1,\ldots,\alpha_n) > 1$. Then a monomial
$f = x_1^{p_1}\cdots x_n^{p_n}$ in $\HOM(V_{\bs{a}}, W_{-b})$ is covariant if and only if $\sum_{i=1}^n p_i a_i = -b$.
If $g$ does not divide $b$, then there are no solutions to this Diophantine equation and hence no covariants so that the
Hilbert series is $0$. If $g$ divides $b$, then clearly $f$ is a covariant if and only if it is a covariant for the representations
$V_{\bs{a}/g}$ with $\bs{a}/g = (a_1/g,\ldots,a_n/g)$ and $W_{-b/g}$, where $V_{\bs{a}/g}$ is now faithful. Hence, we may assume
that $V$ is faithful with no loss of generality. In addition, the existence of a trivial subrepresentation $\epsilon_{a_i}$ with $a_i = 0$ has the trivial effect
of multiplying the Hilbert series by $1/(1 - t_i)$, and we may assume with no loss of generality that $0$ does not appear as a weight.
\end{remark}

\begin{theorem}[Maximally graded Hilbert series of $\Sp^1$-covariants]
\label{thrm:S1Max}
Let $V\simeq \C^n$ be a representation of $\Sp^1$ with weight vector
$\bs{a} = (-\alpha_1,\ldots,-\alpha_k,\alpha_{k+1},\ldots,\alpha_n)$ where $n \geq 1$ and each $\alpha_i > 0$,
and let $W\simeq \C$ be the
irreducible representation of the circle with weight $-b$. Let $\bs{t} = (t_1,\ldots,t_n)$, and let
$\Hilb_{\bs{a};b}^{\Sp^1}(\bs{t})$ denote the Hilbert series for the maximal $\N^n$-grading of the
covariants $\HOM(V,W)^{\Sp^1} = (\C[V]\otimes W)^{\Sp^1}$. Then
\begin{equation}
    \label{eq:S1Max}
    \Hilb_{\bs{a};b}^{\Sp^1}(\bs{t})
    =
    \sum_{\substack{i=1 \\ \zeta^{\alpha_i} = 1}}^k  \frac{\zeta^b t_i^{b/\alpha_i}}{ \alpha_i
        \prod\limits_{\substack{j=1 \\ j\neq i}}^n (1 - \zeta^{a_j} t_j t_i^{- a_j/a_i}) }
        + (-1)^k \sum\limits_{\bs{y}\in\mathcal{S}_{\bs{a},b}}
        \prod\limits_{i=1}^k t_i^{- y_i - 1}
        \prod\limits_{i=k+1}^n t_i^{y_i},
\end{equation}
where the sum $\displaystyle \sum_{\substack{i=1 \\ \zeta^{\alpha_i} = 1}}^k$ is shorthand for $\displaystyle \sum_{i=1}^k \sum_{\zeta^{\alpha_i} = 1}$,
summing over $i = 1,\ldots,k$ and all $\alpha_i$th roots of unity $\zeta$, and
\begin{equation}
\label{eq:S1MaxS}
    \mathcal{S}_{\bs{a},b}  =   \left\{ \bs{y} \in \N^n :
            \sum\limits_{j=1}^n y_j \alpha_j = - b - \sum_{j=1}^k \alpha_j \right\},
\end{equation}
which is empty if $b + \sum_{j=1}^k \alpha_j > 0$.
\end{theorem}
\begin{proof}
Using the Molien-Weyl formula,
\begin{align*}
    \Hilb_{\bs{a};b}^{\Sp^1}(\bs{t})
    &=      \frac{1}{2\pi\sqrt{-1}} \int_{\Sp^1} \frac{z^b \, dz}{z
                \prod\limits_{j=1}^k (1 - t_jz^{-\alpha_j})
                \prod\limits_{j=k+1}^n (1 - t_jz^{\alpha_j}) }
    \\&=    \frac{1}{2\pi\sqrt{-1}} \int_{\Sp^1} \frac{z^{b - 1 + \sum_{j=1}^k \alpha_j}\,dz}{
                \prod\limits_{j=1}^k (z^{\alpha_j} - t_j)
                \prod\limits_{j=k+1}^n (1 - t_jz^{\alpha_j}) }
\end{align*}
for $\lvert t_i\rvert < 1$.
If $b - 1 + \sum_{j=1}^k \alpha_j \geq 0$, i.e., $- b - \sum_{j=1}^k \alpha_j < 0$,
then the integrand is holomorphic at $z = 0$; note that in this case $\mathcal{S}_{\bs{a},b} = \emptyset$.
If $b - 1 + \sum_{j=1}^k \alpha_j < 0$, then there is a pole
at $z = 0$. In this case, we compute the residue at $z = 0$ as follows. We consider the integrand as the product
of factors of the form $z^{b - 1 + \sum_{j=1}^k \alpha_j}$,
\[
    \frac{1}{z^{\alpha_j} - t_j}
        =   -\frac{1}{t_j} - \frac{z^{\alpha_j}}{t_j^2} - \frac{z^{2\alpha_j}}{t_j^3} - \cdots
        =   \sum\limits_{d=0}^\infty - t_j^{- d - 1} z^{d \alpha_j},
\]
and
\[
    \frac{1}{1 - t_j z^{\alpha_j}}
        =   1 + t_j z^{\alpha_j} + t_j^2 z^{2\alpha_j} + \cdots
        =   \sum\limits_{d=0}^\infty t_j^d z^{d \alpha_j}.
\]
Via the Cauchy product formula, for every collection of nonnegative integers $\bs{y} =(y_1,\ldots,y_n)$ such that
$y_1 \alpha_1 + \cdots + y_n \alpha_n + \big(b - 1 + \sum_{j=1}^k \alpha_j\big) = -1$, i.e.,
$y_1 \alpha_1 + \cdots + y_n \alpha_n = - b - \sum_{j=1}^k \alpha_j$,
there is a corresponding term contributing to the $z^{-1}$ term of the product series given by
\[
    (-1)^k \prod\limits_{i=1}^k t_i^{- y_i - 1}
        \prod\limits_{i=k+1}^n t_i^{y_i}.
\]
Summing over all such $\bs{y}$, the residue at $z = 0$ is given by
\[
    (-1)^k \sum\limits_{\bs{y}\in\mathcal{S}_{\bs{a},b}}
        \prod\limits_{i=1}^k t_i^{- y_i - 1}
        \prod\limits_{i=k+1}^n t_i^{y_i}.
\]

The other poles (and the only poles when $b - 1 + \sum_{j=1}^k \alpha_j \geq 0$) occur when $z^{\alpha_i} = t_i$ for some $i\in \{1,\ldots,k\}$, i.e.,
$z = \zeta t_i^{1/\alpha_i}$ where $\zeta$
is an $\alpha_i$th root of unity and the $t_i^{1/\alpha_i}$ are defined with respect to a fixed branch of the logarithm.
Note that we can choose this branch of $\log$ so as to contain all $t_i$ in its domain. Moreover, each of these poles is simple
for generic choice of $t_i$, e.g., by assuming that the moduli of the $t_i$ are distinct.

Fixing a specific $i \leq k$ and $\alpha_i$th root $\zeta_0$, we express the integrand as
\[
    \frac{z^{b - 1 + \sum_{j=1}^k \alpha_j}}{
        (z^{\alpha_i} - t_i) \prod\limits_{\substack{j=1 \\ j\neq i}}^k (z^{\alpha_j} - t_j)
        \prod\limits_{j=k+1}^n (1 - t_jz^{\alpha_j}) }
    =
    \frac{z^{b - 1 + \sum_{j=1}^k \alpha_j}}{
        (z - \zeta_0 t_i^{1/\alpha_i}) \prod\limits_{\substack{ \zeta^{\alpha_i} = 1\\ \zeta\neq\zeta_0}} (z - \zeta t_i^{1/\alpha_i})
        \prod\limits_{\substack{j=1 \\ j\neq i}}^k (z^{\alpha_j} - t_j)
        \prod\limits_{j=k+1}^n (1 - t_jz^{\alpha_j}) }.
\]
Then the residue at $\zeta_0 t_i^{1/\alpha_i}$ is given by
\begin{align}
    \nonumber
    &\frac{(\zeta_0 t_i^{1/\alpha_i})^{b - 1 + \sum_{j=1}^k \alpha_j}}{
        \prod\limits_{\substack{\zeta^{\alpha_i} = 1\\ \zeta\neq\zeta_0}} (\zeta_0 t_i^{1/\alpha_i} - \zeta t_i^{1/\alpha_i})
        \prod\limits_{\substack{j=1 \\ j\neq i}}^k ((\zeta_0 t_i^{1/\alpha_i})^{\alpha_j} - t_j)
        \prod\limits_{j=k+1}^n (1 - t_j (\zeta_0 t_i^{1/\alpha_i})^{\alpha_j}) }
    \\&\nonumber
    \quad =
    \frac{\zeta_0^{b - 1 + \sum_{j=1}^k \alpha_j} t_i^{(b - 1 + \sum_{j=1}^k \alpha_j)/\alpha_i}}{
        \zeta_0^{\alpha_i - 1} t_i^{(\alpha_i - 1)/\alpha_i}
        \prod\limits_{\substack{\zeta^{\alpha_i} = 1\\ \zeta\neq 1}} (1 - \zeta)
        \prod\limits_{\substack{j=1 \\ j\neq i}}^k (\zeta_0^{\alpha_j} t_i^{\alpha_j/\alpha_i} - t_j)
        \prod\limits_{j=k+1}^n (1 - t_j \zeta_0^{\alpha_j} t_i^{\alpha_j/\alpha_i} ) }
    \\&\label{eq:S1MaxLastTermExpr}
    \quad =
    \frac{\zeta_0^b t_i^{b/\alpha_i}}{ \alpha_i
        \prod\limits_{\substack{j=1 \\ j\neq i}}^k (1 - \zeta_0^{-\alpha_j} t_j t_i^{- \alpha_j/\alpha_i})
        \prod\limits_{j=k+1}^n (1 - \zeta_0^{\alpha_j} t_j t_i^{\alpha_j/\alpha_i} ) }.
\end{align}
Summing over all $i \leq k$ and $\zeta$ completes the proof.
\end{proof}

\begin{remark}
\label{rem:S1MaxSWeightsAllPos}
Note that Theorem~\ref{thrm:S1Max} applies in the somewhat trivial cases where the weights of $\bs{a}$ all have the same
sign. In this case, a monomial $x_1^{p_1}\cdots x_n^{p_n}$ is a covariant if and only if the $p_i$ are solutions to
$\sum_{i=1}^n p_i \alpha_i = \pm b$. As there are clearly only finitely many solutions, the covariants are always a
finite-dimensional vector space, which may be $0$ or only contain the constants.

If all weights are positive, then the first sum is empty and the Hilbert series is given by the second sum.
If $b = 0$, the second sum is $1$, as $\mathcal{S}_{\bs{a},0}$ contains only the zero solution, yielding a Hilbert series
of $1$. If $b > 0$, the second sum is $0$ as $\mathcal{S}_{\bs{a},b} = \emptyset$, yielding a Hilbert series of $0$.
If $b < 0$, then the second sum is (usually) nontrivial but will yield a polynomial with no constant term, as the constants
are clearly not covariant.

Now assume all weights are negative. If $b > -\sum_{i=1}^n\alpha_i$, then $\mathcal{S}_{\bs{a},b} = \emptyset$ so the Hilbert series
is given by the first sum. If $b = 0$, then the first sum is nontrivial but will always equal $1$, as only the constants are
invariant. If $b > 0$, then the first sum is nontrivial and will yield a polynomial.
If $b < 0$, then one easily checks from the definition that there are no covariants;
if $-\sum_{i=1}^n\alpha_i < b < 0$, then the first sum is nontrivial but will always yield $0$,
and if $b \leq -\sum_{i=1}^n\alpha_i$, then both sums can be nontrivial but the resulting Hilbert series is $0$.
\end{remark}

\begin{remark}
\label{rem:S1MaxSEmpty}
Note that $\Hilb_{\bs{a};b}^{\Sp^1}(\bs{t}) = \Hilb_{-\bs{a};-b}^{\Sp^1}(\bs{t})$ up to permuting the
$t_i$ to maintain the convention that negative weights appear first in $\bs{a}$. Hence, in applications
of Theorem~\ref{thrm:S1Max}, one may always reduce to a case where $b \geq 0$ (and, if $b = 0$ and all weights have the
same sign, all weights are negative) so that $\mathcal{S}_{\bs{a},b} = \emptyset$.
\end{remark}

\begin{remark}
\label{rem:S1MaxCovarViaInvar}
The Hilbert series of covariants can be computed in terms of the Hilbert series of invariants
by observing that the covariants $\HOM(V_{\bs{a}},W_{-b})^{\Sp^1}$ are precisely the invariants
of the representation with weight vector $(\bs{a},b)$ that are linear in the variable corresponding
to the weight $b$. Hence, up to permuting the $t_i$ to maintain our convention about ordering the
weights, $\Hilb_{\bs{a};b}^{\Sp^1}(\bs{t})$ is the degree $1$ coefficient of the series expansion
of $\Hilb_{(\bs{a},b);0}^{\Sp^1}(\bs{t})$ at $t_{n+1} = 0$ where $t_{n+1}$ is the variable associated
to the last weight $b$. In practice, adding the additional weight $b$ increases the computation time
so that applying Theorem~\ref{thrm:S1Max} directly is faster.
\end{remark}

We give the following examples to illustrate Theorem~\ref{thrm:S1Max}.

\begin{example}
\label{ex:S1Max1}
Consider the invariants (i.e., $b = 0$) of the representation $V_{\bs{a}}$ with weight vector $\bs{a} = (-1, -2, 1, 2)$.
Applying Theorem~\ref{thrm:S1Max} via the algorithm described in Section~\ref{sec:Algorithms},
the Hilbert series is given by
\begin{equation}
\label{eq:exS1Max}
    \Hilb_{(-1,-2,1,2),0}^{\Sp^1}(t_1,t_2,t_3,t_4)  =
    \frac{1 - t_1^2 t_2 t_3^2 t_4}
        {(1 - t_1 t_3) (1 - t_2 t_3^2)(1 - t_1^2 t_4)(1 - t_2 t_4)}.
\end{equation}
Lisiting all terms with exponents $\le 5$ in colexicographic order,
the Taylor expansion at $\bs{t} = (0,0,0,0)$ begins
\begin{align*}
    &1+t_1 t_3+t_1^2 t_3^2+t_2 t_3^2+t_1^3 t_3^3+t_1 t_2 t_3^3+t_1^4 t_3^4+t_1^2 t_2 t_3^4+t_2^2 t_3^4+t_1^5 t_3^5+t_1^3
   t_2 t_3^5+t_1 t_2^2 t_3^5+t_1^2 t_4+t_2 t_4+t_1^3 t_3 t_4
   \\&\quad
   +t_1 t_2 t_3 t_4+t_1^4 t_3^2 t_4+t_1^2 t_2 t_3^2 t_4+t_2^2t_3^2 t_4+t_1^5 t_3^3 t_4+t_1^3 t_2 t_3^3 t_4+t_1 t_2^2 t_3^3 t_4+t_1^4 t_2 t_3^4 t_4+t_1^2 t_2^2 t_3^4 t_4
   +t_2^3t_3^4 t_4
   \\&\quad
   +t_1^5 t_2 t_3^5 t_4+t_1^3 t_2^2 t_3^5 t_4+t_1 t_2^3 t_3^5 t_4+t_1^4 t_4^2+t_1^2 t_2t_4^2+t_2^2t_4^2+t_1^5 t_3t_4^2
   +t_1^3 t_2 t_3 t_4^2+t_1 t_2^2 t_3 t_4^2+t_1^4 t_2 t_3^2 t_4^2+t_1^2 t_2^2 t_3^2 t_4^2
   \\&\quad
   +t_2^3t_3^2 t_4^2+t_1^5 t_2 t_3^3t_4^2+t_1^3 t_2^2 t_3^3 t_4^2+t_1 t_2^3 t_3^3 t_4^2+t_1^4 t_2^2 t_3^4 t_4^2+t_1^2 t_2^3t_3^4 t_4^2+t_2^4 t_3^4 t_4^2+t_1^5 t_2^2 t_3^5 t_4^2+t_1^3 t_2^3 t_3^5 t_4^2+t_1 t_2^4 t_3^5 t_4^2
   \\&\quad
   +t_1^4 t_2t_4^3+t_1^2 t_2^2 t_4^3+t_2^3 t_4^3+t_1^5 t_2 t_3 t_4^3+t_1^3 t_2^2 t_3 t_4^3+t_1 t_2^3 t_3 t_4^3+t_1^4 t_2^2 t_3^2t_4^3+t_1^2 t_2^3 t_3^2 t_4^3+t_2^4 t_3^2 t_4^3+t_1^5 t_2^2 t_3^3 t_4^3+t_1^3 t_2^3 t_3^3 t_4^3
   \\&\quad
   +t_1 t_2^4 t_3^3t_4^3+t_1^4 t_2^3 t_3^4 t_4^3+t_1^2 t_2^4 t_3^4 t_4^3+t_2^5 t_3^4 t_4^3+t_1^5 t_2^3 t_3^5 t_4^3+t_1^3 t_2^4 t_3^5t_4^3+t_1 t_2^5 t_3^5 t_4^3+t_1^4 t_2^2 t_4^4+t_1^2 t_2^3 t_4^4+t_2^4 t_4^4+t_1^5 t_2^2 t_3 t_4^4
   \\&\quad
   +t_1^3 t_2^3 t_3t_4^4+t_1 t_2^4 t_3 t_4^4+t_1^4 t_2^3 t_3^2 t_4^4+t_1^2 t_2^4 t_3^2 t_4^4+t_2^5 t_3^2 t_4^4+t_1^5 t_2^3 t_3^3t_4^4+t_1^3 t_2^4 t_3^3 t_4^4+t_1 t_2^5 t_3^3 t_4^4+t_1^4 t_2^4 t_3^4 t_4^4+t_1^2 t_2^5 t_3^4 t_4^4
   \\&\quad
   +t_1^5 t_2^4t_3^5 t_4^4+t_1^3 t_2^5 t_3^5 t_4^4+t_1^4 t_2^3 t_4^5+t_1^2 t_2^4 t_4^5+t_2^5 t_4^5+t_1^5 t_2^3 t_3 t_4^5
   +t_1^3t_2^4 t_3 t_4^5+t_1 t_2^5 t_3 t_4^5+t_1^4 t_2^4 t_3^2 t_4^5+t_1^2 t_2^5 t_3^2 t_4^5+t_1^5 t_2^4 t_3^3 t_4^5
   \\&\quad
   +t_1^3t_2^5 t_3^3 t_4^5+t_1^4 t_2^5 t_3^4 t_4^5+t_1^5 t_2^5 t_3^5 t_4^5 + \cdots.
\end{align*}
Hence, each coefficient is either $0$ or $1$, and each nonzero term can be though of as an invariant in the variables
$(t_1,t_2,t_3,t_4)$. Via this interpretation, the above expansion yields a complete list of the invariants of
$\N^4$-degree $(d_1,d_2,d_3,d_4)$ with each $d_i \leq 5$. That is, the Hilbert series
$\Hilb_{(-1,-2,1,2),0}^{\Sp^1}(t_1,t_2,t_3,t_4)$ completely determines the algebra of invariants.

Note that in this case, the structure of the algebra of invariants is clear from Equation~\eqref{eq:exS1Max}.
In particular, the algebra of invariants in this case is generated by
the monomials $t_1 t_3$, $t_2 t_3^2$, $t_1^2 t_4$, and $t_2 t_4$ with the single relation
$(t_1 t_3)^2 (t_2 t_4) - (t_2 t_3^2) (t_1^2 t_4)$, see
\cite[Chapter~I, Theorem~2.3 and Corollary~3.8]{StanleyCombComAlg}.

Letting $V_{\bs{a}}$ be as above and $b = 2$, Theorem~\ref{thrm:S1Max} and Section~\ref{sec:Algorithms} yield
\[
    \Hilb_{(-1,-2,1,2),2}^{\Sp^1}(t_1,t_2,t_3,t_4)  =
    \frac{t_1^2+t_2-t_1^2 t_2 t_3^2-t_1^2 t_2 t_4}
        {(1 - t_1 t_3)(1 - t_2 t_3^2)(1 - t_1^2 t_4)(1 - t_2 t_4)},
\]
whose expansion begins
\begin{align*}
   &t_1^2+t_2+t_1^3 t_3+t_1 t_2 t_3+t_1^4 t_3^2+t_1^2 t_2 t_3^2+t_2^2 t_3^2+t_1^5 t_3^3+t_1^3 t_2 t_3^3+t_1 t_2^2
   t_3^3+t_1^4 t_2 t_3^4+t_1^2 t_2^2 t_3^4+t_2^3 t_3^4+t_1^5 t_2 t_3^5+t_1^3 t_2^2 t_3^5
   \\&\quad
   +t_1 t_2^3 t_3^5+t_1^4
   t_4+t_1^2 t_2 t_4+t_2^2 t_4+t_1^5 t_3 t_4+t_1^3 t_2 t_3 t_4+t_1 t_2^2 t_3 t_4+t_1^4 t_2 t_3^2 t_4+t_1^2 t_2^2 t_3^2
   t_4+t_2^3 t_3^2 t_4+t_1^5 t_2 t_3^3 t_4
   \\&\quad
   +t_1^3 t_2^2 t_3^3 t_4+t_1 t_2^3 t_3^3 t_4+t_1^4 t_2^2 t_3^4 t_4+t_1^2 t_2^3
   t_3^4 t_4+t_2^4 t_3^4 t_4+t_1^5 t_2^2 t_3^5 t_4+t_1^3 t_2^3 t_3^5 t_4+t_1 t_2^4 t_3^5 t_4+t_1^4 t_2 t_4^2+t_1^2
   t_2^2 t_4^2+t_2^3 t_4^2
   \\&\quad
   +t_1^5 t_2 t_3 t_4^2+t_1^3 t_2^2 t_3 t_4^2+t_1 t_2^3 t_3 t_4^2+t_1^4 t_2^2 t_3^2 t_4^2+t_1^2
   t_2^3 t_3^2 t_4^2+t_2^4 t_3^2 t_4^2+t_1^5 t_2^2 t_3^3 t_4^2+t_1^3 t_2^3 t_3^3 t_4^2+t_1 t_2^4 t_3^3 t_4^2+t_1^4
   t_2^3 t_3^4 t_4^2
   \\&\quad
   +t_1^2 t_2^4 t_3^4 t_4^2+t_2^5 t_3^4 t_4^2+t_1^5 t_2^3 t_3^5 t_4^2+t_1^3 t_2^4 t_3^5 t_4^2+t_1
   t_2^5 t_3^5 t_4^2+t_1^4 t_2^2 t_4^3+t_1^2 t_2^3 t_4^3+t_2^4 t_4^3+t_1^5 t_2^2 t_3 t_4^3+t_1^3 t_2^3 t_3 t_4^3+t_1
   t_2^4 t_3 t_4^3
   \\&\quad
   +t_1^4 t_2^3 t_3^2 t_4^3+t_1^2 t_2^4 t_3^2 t_4^3+t_2^5 t_3^2 t_4^3+t_1^5 t_2^3 t_3^3 t_4^3+t_1^3
   t_2^4 t_3^3 t_4^3+t_1 t_2^5 t_3^3 t_4^3+t_1^4 t_2^4 t_3^4 t_4^3+t_1^2 t_2^5 t_3^4 t_4^3+t_1^5 t_2^4 t_3^5
   t_4^3+t_1^3 t_2^5 t_3^5 t_4^3
   \\&\quad
   +t_1^4 t_2^3 t_4^4+t_1^2 t_2^4 t_4^4+t_2^5 t_4^4+t_1^5 t_2^3 t_3 t_4^4+t_1^3 t_2^4 t_3
   t_4^4+t_1 t_2^5 t_3 t_4^4+t_1^4 t_2^4 t_3^2 t_4^4+t_1^2 t_2^5 t_3^2 t_4^4+t_1^5 t_2^4 t_3^3 t_4^4+t_1^3 t_2^5 t_3^3
   t_4^4+t_1^4 t_2^5 t_3^4 t_4^4
   \\&\quad
   +t_1^5 t_2^5 t_3^5 t_4^4+t_1^4 t_2^4 t_4^5+t_1^2 t_2^5 t_4^5+t_1^5 t_2^4 t_3
   t_4^5+t_1^3 t_2^5 t_3 t_4^5+t_1^4 t_2^5 t_3^2 t_4^5+t_1^5 t_2^5 t_3^3 t_4^5 + \cdots.
\end{align*}
\end{example}

\begin{example}
\label{ex:S1Max2}
To illustrate Remark~\ref{rem:S1MaxSEmpty}, consider the representation $V_{\bs{a}}$ with weight vector
$\bs{a} = (-2, 2, 5)$ and $b = -4$. Applying Theorem~\ref{thrm:S1Max} directly, the first sum in Equation~\eqref{eq:S1Max}
is equal to
\[
    \frac{1}{t_1^2 (1 - t_1 t_2)(1 - t_1^5 t_3^2)},
\]
the set $\mathcal{S}_{(-2,2,5),-4} = \{ (1,0,0); (0,1,0) \}$, and so the second sum is equal to
\[
    -\frac{1}{t_1^2}-\frac{t_2}{t_1},
\]
yielding a Hilbert series of
\[
    \Hilb_{(-2,2,5),-4}^{\Sp^1}(t_1,t_2,t_3)  =
    \frac{t_2^2 + t_1^3 t_3^2 - t_1^5 t_2^2 t_3^2}
        {(1 - t_1 t_2)(1 - t_1^5 t_3^2)}.
\]
As noted in Remark~\ref{rem:S1MaxSEmpty}, consideration of $\mathcal{S}_{(-2,2,5),-4}$ can be avoided by considering
instead $-\bs{a}$ and $-b$ (permuting coordinates to conform to our sign conventions). In this case, the second sum in
Equation~\eqref{eq:S1Max} is empty and the first yields
\[
    \Hilb_{(-2,-5,2),4}^{\Sp^1}(t_1,t_2,t_3)  =
    \frac{t_1^2 + t_2^2 t_3^3 - t_1^2 t_2^2 t_3^5}
        {(1 - t_1 t_3)(1 - t_2^2 t_3^5)},
\]
which yields the same result up to permuting variables.
\end{example}

\begin{example}
\label{ex:S1Max3}
To illustrate Remark~\ref{rem:S1MaxCovarViaInvar}, let $\bs{a} = (-1,2,3)$ and $b = 3$. The second sum in
Equation~\eqref{eq:S1Max} is empty, and the first yields
\[
    \Hilb_{(-1,2,3),3}^{\Sp^1}(t_1,t_2,t_3)  =
    \frac{t_1^3}{(1 - t_1^2 t_2)(1 - t_1^3 t_3)}.
\]
Alternatively, one can compute the Hilbert series of the invariants of the representation with weight matrix
$(-1,2,3,3)$, yielding
\[
    \Hilb_{(-1,2,3,3),0}^{\Sp^1}(t_1,t_2,t_3,t_4)  =
    \frac{1}{(1 - t_1^2 t_2)(1 - t_1^3 t_3)(1 - t_1^3 t_4)}.
\]
Expanding at $t_4 = 0$ and extracting the coefficient of $t_4$ recovers $\Hilb_{(-1,2,3),3}^{\Sp^1}(t_1,t_2,t_3)$.
\end{example}

Substituting $t_i = t$ in the expression in Theorem~\ref{thrm:S1Max} and applying the analytic continuation argument of
\cite[Section~3.3]{HerbigHerdenSeatonT2} in the case when the $\alpha_i$ for $1\leq i\leq k$ are not distinct
yields the unviariate Hilbert series of the covariants. In particular, setting $b = 0$, this recovers the expression for
the Hilbert series of the invariants in \cite[Theorem~3.3]{CowieHerbigSeatonHerden}.

\begin{corollary}[Univariate Hilbert series of $\Sp^1$-covariants]
\label{cor:S1Univar}
Let $V\simeq \C^n$ be a representation of $\Sp^1$ with weight vector
$\bs{a} = (-\alpha_1,\ldots,-\alpha_k,\alpha_{k+1},\ldots,\alpha_n)$ where each $\alpha_i > 0$, and let $W\simeq \C$ be the
irreducible representation of the circle with weight $-b$. The univariate Hilbert series $\Hilb_{\bs{a};b}^{\Sp^1}(t)$ of
the module of covariants $\HOM(V,W)^{\Sp^1} = (\C[V]\otimes W)^{\Sp^1}$ is given by
\begin{equation}
    \label{eq:S1Univar}
    \Hilb_{\bs{a};b}^{\Sp^1}(t)
    =
    \lim\limits_{(c_1,\ldots,c_n)\to\bs{a}}
    \sum_{\substack{i=1 \\ \zeta^{\alpha_i} = 1}}^k  \frac{\zeta^b t^{b/\alpha_i}}{ \alpha_i
        \prod\limits_{\substack{j=1 \\ j\neq i}}^n (1 - \zeta^{a_j} t^{(c_i - c_j)/c_i}) }
        + (-1)^k \sum\limits_{\bs{y}\in\mathcal{S}_{\bs{a},b}}
            t^{- \sum_{i=1}^k (y_i + 1) + \sum_{i=k+1}^n y_i},
\end{equation}
where $\mathcal{S}_{\bs{a},b}$ is as in Theorem~\ref{thrm:S1Max}.
\end{corollary}

% xxxxxxxxxxxxxxxxxxxxxxxxxxxxxxxxxxxxxxxxxxxxxxxxxxxxxxxxxxxxxxxxxxxxxxxxx

\subsection{The $\N^2$-grading of a cotangent-lifted representation and the symplectic quotient}
\label{subsec:S1Cot}

Suppose $V \simeq \C^{2n}$ is a cotangent-lifted representation of $\Sp^1$, meaning that the
weight vector is of the form $\bs{a}^c = (\bs{a}, -\bs{a})$ for a weight vector
$\bs{a}\in\Z^n$. In this case, we use the notation
\[
    \bs{a}^c=   (a_1,\ldots,a_n,-a_1,\ldots,-a_n)
            =   (-\alpha_1,\ldots,-\alpha_k,\alpha_{k+1},\ldots,\alpha_n,
                    \alpha_1,\ldots,\alpha_k,-\alpha_{k+1},\ldots,-\alpha_n),
\]
continuing the convention that each $\alpha_j = \lvert a_j\rvert$.
We again assume that no weight is zero, i.e., that $V$ has no trivial factors, to avoid trivialities.

We consider the decomposition $V_{\bs{a}^c} = V_1\oplus V_2$ where $V_1$ has weight vector $\bs{a}$
and $V_2$ has weight vector $-\bs{a}$ and let $(d_1,d_2)$ denote the corresponding grading
with formal variables $(s,t)$. Then we have the following.

\begin{corollary}[Bivariate Hilbert series of $\Sp^1$-covariants of a cotangent-lifted representation]
\label{cor:S1CotanBivar}
Let $V\simeq \C^n$ be a representation of $\Sp^1$ with weight vector
$\bs{a} = (-\alpha_1,\ldots,-\alpha_k,\alpha_{k+1},\ldots,\alpha_n)$ where each $\alpha_i > 0$,
let $\bs{a}^c = (\bs{a}, -\bs{a})$ denote the weight vector of the cotangent lift of $V$, and let $W\simeq \C$ be the
irreducible representation of the circle with weight $-b$. Let
$\Hilb_{\bs{a}^c;b}^{\Sp^1}(s,t)$ denote the bivariate Hilbert series of the
covariants ${\HOM(V\oplus V^\ast,W)^{\Sp^1}} = (\C[V\oplus V^\ast]\otimes W)^{\Sp^1}$. Then
\begin{multline}
\label{eq:S1CotanBivarDegen}
    \Hilb_{\bs{a}^c;b}^{\Sp^1}(s,t) = \lim\limits_{(c_1,\ldots,c_n)\to\bs{a}}
    \frac{1}{1 - st}\left(\sum_{\substack{i=1 \\ \zeta^{\alpha_i} = 1}}^k  \frac{\zeta^b s^{b/\alpha_i}}{ \alpha_i
        \prod\limits_{\substack{j=1 \\ j\neq i}}^n
            (1 - \zeta^{a_j} s^{(c_i - c_j)/c_i} )
            (1 - \zeta^{-a_j} s^{c_j/c_i} t) }
    \right.\\ \left.
    + \sum_{\substack{i=k+1 \\ \zeta^{\alpha_i} = 1}}^n \frac{\zeta^b t^{b/\alpha_i}}{ \alpha_i
        \prod\limits_{\substack{j=1 \\ j\neq i}}^n
            (1 - \zeta^{a_j} s t^{c_j/c_i} )
            (1 - \zeta^{-a_j} t^{(c_i - c_j)/c_i} ) }\right)
    + (-1)^n \sum\limits_{(\bs{y}, \bs{y'})\in\mathcal{S}_{\bs{a}^c,b}}
        \prod\limits_{i=1}^k s^{- y_i - 1} t^{y_i^\prime}
        \prod\limits_{i=k+1}^n s^{y_i} t^{- y_i^\prime - 1},
\end{multline}
where
\begin{equation}
\label{eq:S1CotanBivarS}
    \mathcal{S}_{\bs{a}^c,b}
        =    \left\{ (\bs{y}, \bs{y'}) \in \N^{2n} :
                \sum\limits_{j=1}^n (y_j + y_j^\prime) \alpha_j
                        = - b  - \sum_{j=1}^n \alpha_j \right\},
\end{equation}
which is empty if $b + \sum\limits_{j=1}^n \alpha_j > 0$.
If the $a_i$ are distinct, then the limit is unnecessary and Equation~\eqref{eq:S1CotanBivarDegen} becomes
\begin{multline*}
    \Hilb_{\bs{a}^c;b}^{\Sp^1}(s,t) =
    \frac{1}{1 - st}\left(\sum_{\substack{i=1 \\ \zeta^{\alpha_i} = 1}}^k  \frac{\zeta^b s^{b/\alpha_i}}{ \alpha_i
        \prod\limits_{\substack{j=1 \\ j\neq i}}^n
            (1 - \zeta^{a_j} s^{(a_i - a_j)/a_i} )
            (1 - \zeta^{-a_j} s^{a_j/a_i} t) }
    \right.\\ \left.
    + \sum_{\substack{i=k+1 \\ \zeta^{\alpha_i} = 1}}^n \frac{\zeta^b t^{b/\alpha_i}}{ \alpha_i
        \prod\limits_{\substack{j=1 \\ j\neq i}}^n
            (1 - \zeta^{a_j} s t^{a_j/a_i} )
            (1 - \zeta^{-a_j} t^{(a_i - a_j)/a_i} ) }\right)
    + (-1)^n \sum\limits_{(\bs{y}, \bs{y'})\in\mathcal{S}_{\bs{a}^c,b}}
        \prod\limits_{i=1}^k s^{- y_i - 1} t^{y_i^\prime}
        \prod\limits_{i=k+1}^n s^{y_i} t^{- y_i^\prime - 1}.
\end{multline*}
\end{corollary}
\begin{proof}
Using the weight vector $\bs{a}^c$ and substituting $t_i = s$
for $1\leq i \leq n$ and $t_i = t$ for $n+1\leq i \leq 2n$, the first sum in Equation~\eqref{eq:S1Max}
becomes
\[
    \sum_{\substack{i=1 \\ \zeta^{\alpha_i} = 1}}^k  \frac{(1 - st)^{-1}\zeta^b s^{b/\alpha_i}}{ \alpha_i
        \prod\limits_{\substack{j=1 \\ j\neq i}}^n
            (1 - \zeta^{a_j} s^{1 - a_j/a_i} )
            (1 - \zeta^{-a_j} s^{a_j/a_i} t) }
    +
    \sum_{\substack{i=k+1 \\ \zeta^{\alpha_i} = 1}}^n \frac{(1 - st)^{-1}\zeta^b t^{b/\alpha_i}}{ \alpha_i
        \prod\limits_{\substack{j=1 \\ j\neq i}}^n
            (1 - \zeta^{a_j} s t^{a_j/a_i} )
            (1 - \zeta^{-a_j} t^{1 - a_j/a_i} ) }.
\]
Note that if the $a_i$ are not distinct, then the factors of the form $(1 - \zeta^{a_j} s^{1 - a_j/a_i} )$
introduce singularities when $a_i = a_j$; in this case, we apply the analytic continuation argument of
\cite[Section~3.3]{HerbigHerdenSeatonT2} and take the limit as parameters $c_i$ approach the $a_i$.
If $b + \sum_{j=1}^n \alpha_j \le 0$, then we specialize Equation~\eqref{eq:S1MaxS} to the case at hand as
Equation~\eqref{eq:S1CotanBivarS}, and the second sum in Equation~\eqref{eq:S1Max} becomes
\begin{align*}
    &(-1)^n \sum\limits_{(\bs{y}, \bs{y'})\in\mathcal{S}_{\bs{a}^c,b}}
        \prod\limits_{i=1}^k s^{- y_i - 1}
        \prod\limits_{i=k+1}^n t^{- y_i^\prime - 1}
        \prod\limits_{i=k+1}^n s^{y_i}
        \prod\limits_{i=1}^k t^{y_i^\prime}
    \\&=(-1)^n \sum\limits_{(\bs{y}, \bs{y'})\in\mathcal{S}_{\bs{a}^c,b}}
        \prod\limits_{i=1}^k s^{- y_i - 1} t^{y_i^\prime}
        \prod\limits_{i=k+1}^n s^{y_i} t^{- y_i^\prime - 1},
\end{align*}
completing the proof.
\end{proof}

\begin{remark}
\label{rem:S1BivarSEmpty}
Using the observation of Remark~\ref{rem:S1MaxSEmpty} and the fact that $\bs{a}^c = -\bs{a}^c$
up to permuting weights, we have $\Hilb_{\bs{a}^c;b}^{\Sp^1}(s,t) = \Hilb_{\bs{a}^c;-b}^{\Sp^1}(t,s)$.
Hence, applications of Corollary~\ref{cor:S1CotanBivar} can always be reduced to a case where
$\mathcal{S}_{\bs{a}^c,b} = \emptyset$.
\end{remark}

Taking the limit as $s\to t$ and, for $1\leq i\leq k$, relabeling $\zeta$ as $\zeta^{-1}$ to permute
terms and express them uniformly, we obtain the following.

\begin{corollary}[Univariate Hilbert series of $\Sp^1$-covariants of a cotangent-lifted representation]
\label{cor:S1CotanUnivar}
Let $V\simeq \C^n$ be a representation of $\Sp^1$ with weight vector
$\bs{a} = (-\alpha_1,\ldots,-\alpha_k,\alpha_{k+1},\ldots,\alpha_n)$ where each $\alpha_i > 0$,
let $\bs{a}^c = (\bs{a}, -\bs{a})$ denote the weight vector of the cotangent lift of $V$, and let $W\simeq \C$ be the
irreducible representation of the circle with weight $-b$. Let
$\Hilb_{\bs{a}^c;b}^{\Sp^1}(t)$ denote the univariate Hilbert series of the
covariants ${\HOM(V\oplus V^\ast,W)^{\Sp^1}} = (\C[V\oplus V^\ast]\otimes W)^{\Sp^1}$. Then
\begin{multline}
\label{eq:S1CotanUnivarDegen}
    \Hilb_{\bs{a}^c;b}^{\Sp^1}(t) = \lim\limits_{(c_1,\ldots,c_n)\to\bs{a}}
    \frac{1}{1 - t^2} \sum_{\substack{i=1 \\ \zeta^{\alpha_i} = 1}}^n
        \frac{\zeta^{a_i b/\alpha_i} t^{b/\alpha_i}}{ \alpha_i
        \prod\limits_{\substack{j=1 \\ j\neq i}}^n
            (1 - \zeta^{-a_j} t^{(c_i - c_j)/c_i} )
            (1 - \zeta^{a_j} t^{(c_i + c_j)/c_i} ) }
    \\
    + (-1)^n \sum\limits_{(\bs{y}, \bs{y'})\in\mathcal{S}_{\bs{a}^c,b}}
        \prod\limits_{i=1}^k t^{y_i^\prime - y_i - 1}
        \prod\limits_{i=k+1}^n  t^{y_i - y_i^\prime - 1},
\end{multline}
where $\mathcal{S}_{\bs{a}^c,b}$ is as defined in Equation~\eqref{eq:S1CotanBivarS}.
If the $\alpha_i$ are distinct, then the limit is unnecessary and the right-hand side of Equation~\eqref{eq:S1CotanUnivarDegen} becomes
\begin{multline*}
    \frac{1}{1 - t^2} \sum_{\substack{i=1 \\ \zeta^{\alpha_i} = 1}}^n
        \frac{\zeta^{a_i b/\alpha_i} t^{b/\alpha_i}}{ \alpha_i
        \prod\limits_{\substack{j=1 \\ j\neq i}}^n
            (1 - \zeta^{-a_j} t^{(a_i - a_j)/a_i} )
            (1 - \zeta^{a_j} t^{(a_i + a_j)/a_i} ) }
    + (-1)^n \sum\limits_{(\bs{y}, \bs{y'})\in\mathcal{S}_{\bs{a}^c,b}}
        \prod\limits_{i=1}^k t^{y_i^\prime - y_i - 1}
        \prod\limits_{i=k+1}^n t^{y_i - y_i^\prime - 1}.
\end{multline*}
\end{corollary}

In the case of the on-shell invariants associated to the symplectic quotient, the moment map associated to $\bs{a}$ in complex coordinates
$\bs{x} = (x_1,\ldots,x_n)$ for $V$ is given by
\[
    J(\bs{x})
    =
    \sum\limits_{i=1}^n a_i x_i \overline{x_i}
\]
and hence has degree $(1,1)$. Then $\Hilb_{\bs{a}}^{\Sp^1,\on}(s,t) = (1 - st)\Hilb_{\bs{a}}^{\Sp^1,\off}(s,t)$,
where $\Hilb_{\bs{a}}^{\Sp^1,\off}(s,t) = \Hilb_{\bs{a}^c;0}^{\Sp^1}(s,t)$ denotes the corresponding Hilbert series of off-shell invariants;
see \cite[Proposition~2.1]{HerbigSeaton}. We therefore have the following.

\begin{corollary}[Bivariate Hilbert series of the on-shell invariants of an $\Sp^1$-symplectic quotient]
\label{cor:S1SymplecticBivar}
Let $V\simeq \C^n$ be a representation of $\Sp^1$ with weight vector
$\bs{a} = (-\alpha_1,\ldots,-\alpha_k,\alpha_{k+1},\ldots,\alpha_n)$ where each $\alpha_i > 0$.
Let $\Hilb_{\bs{a}}^{\Sp^1,\on}(s,t)$ denote the bivariate Hilbert series of the
on-shell invariants of the corresponding symplectic quotient. Then
\begin{multline}
\label{eq:S1SymplecticBivarDegen}
    \Hilb_{\bs{a}}^{\Sp^1,\on}(s,t) = \lim\limits_{(c_1,\ldots,c_n)\to\bs{a}}
    \left(\sum_{\substack{i=1 \\ \zeta^{\alpha_i} = 1}}^k  \frac{1}{ \alpha_i
        \prod\limits_{\substack{j=1 \\ j\neq i}}^n
            (1 - \zeta^{a_j} s^{(c_i - c_j)/c_i} )
            (1 - \zeta^{-a_j} s^{c_j/c_i} t) }
    \right.\\ \left.
    + \sum_{\substack{i=k+1 \\ \zeta^{\alpha_i} = 1}}^n \frac{1}{ \alpha_i
        \prod\limits_{\substack{j=1 \\ j\neq i}}^n
            (1 - \zeta^{a_j} s t^{c_j/c_i} )
            (1 - \zeta^{-a_j} t^{(c_i - c_j)/c_i} ) }\right).
\end{multline}
If the $a_i$ are distinct, then $\Hilb_{\bs{a}}^{\Sp^1,\on}(s,t)$ is given by
\begin{equation}
\label{eq:S1SymplecticBivarGeneric}
    \sum_{\substack{i=1 \\ \zeta^{\alpha_i} = 1}}^k  \frac{1}{ \alpha_i
        \prod\limits_{\substack{j=1 \\ j\neq i}}^n
            (1 - \zeta^{a_j} s^{(a_i - a_j)/a_i} )
            (1 - \zeta^{-a_j} s^{a_j/a_i} t) }
    + \sum_{\substack{i=k+1 \\ \zeta^{\alpha_i} = 1}}^n \frac{1}{ \alpha_i
        \prod\limits_{\substack{j=1 \\ j\neq i}}^n
            (1 - \zeta^{a_j} s t^{a_j/a_i} )
            (1 - \zeta^{-a_j} t^{(a_i - a_j)/a_i} ) }.
\end{equation}
\end{corollary}

Taking the limit as $s\to t$ and again relabeling $\zeta$ as $\zeta^{-1}$ for $1\leq i\leq k$,
we recover the univariate Hilbert series of \cite[Theorem~3.1]{HerbigSeaton}.

% xxxxxxxxxxxxxxxxxxxxxxxxxxxxxxxxxxxxxxxxxxxxxxxxxxxxxxxxxxxxxxxxxxxxxxxxx

\subsection{Laurent coefficients of the univariate Hilbert series}
\label{subsec:S1Laurent}

Let $V = V_{\bs{a}}$ be a representation of $\Sp^1$ with $\bs{a}\in\Z^n$, let $W = W_{-b}$ be an irreducible
representation of $\Sp^1$, and let $\Hilb_{\bs{a};b}^{\Sp^1}(t)$ denote the univariate Hilbert series as computed
in Corollary~\ref{cor:S1Univar}. We consider the Laurent expansion at $t = 1$ of the form
\[
    \Hilb_{\bs{a};b}^{\Sp^1}(t)
    =   \sum\limits_{m=0}^\infty \gamma_m^{\Sp^1}(\bs{a};b) (1 - t)^{m - n + 1}.
\]
We will use the simplified notation $\gamma_m = \gamma_m^{\Sp^1}(\bs{a};b)$ when the representations are clear from the
context. Note that we index the $\gamma_m$ so that $\gamma_0$ is the coefficient of degree $1 - n$; the pole order
of $\Hilb_{\bs{a};b}^{\Sp^1}(t)$ is at most $n - 1$ and usually obtains this value, but it is possible that $\gamma_0 = 0$,
e.g., if all weights have the same sign.
Our goal in this section is to compute $\gamma_0$ and $\gamma_1$ in general as well as when $V$ is a cotangent lift.
When $b = 0$, these coefficients along with $\gamma_2$ and $\gamma_3$ were computed in \cite[Section~5]{HerbigSeaton}
and \cite[Section~6]{CowieHerbigSeatonHerden}; as we will see, the coefficients in the case of $b\neq 0$ are not
substantially different.

It is difficult to uniformly treat the cases where all weights have the same sign, so let us first consider these cases
separately; see Remark~\ref{rem:S1MaxSWeightsAllPos}.
If all weights have the same sign and $b = 0$, then $\Hilb_{\bs{a};b}^{\Sp^1}(t) = 1$. Hence $\gamma_{n-1} = 1$ and
all other $\gamma_m = 0$. If all weights are positive and $b > 0$, then $\Hilb_{\bs{a};b}^{\Sp^1}(t) = 0$, so all
$\gamma_m$ vanish. If all weights are positive and $b < 0$, then $\Hilb_{\bs{a};b}^{\Sp^1}(t)$ is a polynomial.
First assume $n = 1$. By Remark~\ref{rem:S1Faithful}, $\Hilb_{(a_1);b}^{\Sp^1}(t) = 0$ if the single weight $a_1$ does
not divide $b$ and otherwise $\Hilb_{(a_1);b}^{\Sp^1}(t) = \Hilb_{(1);b/a_1}^{\Sp^1}(t)=t^{-b/a_1}$. In the former case,
all $\gamma_m$ vanish, while in the latter case, $\gamma_0 = 1$ and $\gamma_1 = b/a_1$. If $n = 2$, then $\gamma_0 = 0$
and $\gamma_1$ is a nonnegative integer that counts the number of solutions to $y_1\alpha_1 + y_2\alpha_2 = -b$.
If $n > 2$, then $\gamma_0 = \gamma_1 = 0$. Cases with all weights negative can be reduced
to those above using the fact that $\Hilb_{\bs{a};b}^{\Sp^1}(t) = \Hilb_{-\bs{a};-b}^{\Sp^1}(t)$.

We now assume that $\bs{a}$ contains at least one weight of each sign so that $n \geq 2$. We assume without loss of
generality that $V$ is faithful so that $\gcd(a_1,\ldots,a_n) = 1$, see Remark~\ref{rem:S1Faithful}, and $b \geq 0$,
see Remark~\ref{rem:S1MaxSEmpty}.

By $\bs{sp}_u(\bs{a})$, we mean the partial Schur-Laurent polynomial, see Equation~\eqref{eq:DefPartialSchur},
where the first set of variables is the set of negative weights and the second is the set of positive weights.
We let $\bs{e}_1(\bs{a}) = \sum_{i=1}^n a_i$ denote the degree $1$ elementary symmetric polynomial in the weights.

\begin{theorem}[$\gamma_0$ and $\gamma_1$ for $\Sp^1$-covariants]
\label{thrm:S1Laurent}
Let $V\simeq \C^n$ be a faithful representation of $\Sp^1$ with weight vector
$\bs{a} = (-\alpha_1,\ldots,-\alpha_k,\alpha_{k+1},\ldots,\alpha_n)$ where each $\alpha_i > 0$,
$n\geq 2$, and $1 \le k < n$, and let $W\simeq \C$ be the irreducible representation of the circle with weight $-b$
where $b\geq 0$. Then the first two Laurent coefficients $\gamma_0^{\Sp^1}(\bs{a};b)$ and $\gamma_1^{\Sp^1}(\bs{a};b)$
of the Hilbert series $\Hilb_{\bs{a};b}^{\Sp^1}(t)$ of covariants $\HOM(V,W)^{\Sp^1} = (\C[V]\otimes W)^{\Sp^1}$ are
given by
\begin{align*}
    \gamma_0^{\Sp^1}(\bs{a};b)
    &=      \frac{ -\bs{sp}_{n-2}(\bs{a}) }{\prod\limits_{p=1}^k\prod\limits_{q=k+1}^n (a_p - a_q) }, \quad\text{and}
    \\
    \gamma_1^{\Sp^1}(\bs{a};b)
    &=      \frac{\big(\bs{e}_1(\bs{a}) - 2b\big)\bs{sp}_{n-3}(\bs{a}) - \bs{sp}_{n-2}(\bs{a})}
                {2\prod\limits_{p=1}^k\prod\limits_{q=k+1}^n (a_p - a_q)}
            + \sum\limits_{j=1}^n \left(\frac{2\big|[a_j]_{g_j}^{-1} b\big|_{g_j} - g_j - 1}{2}\right) \gamma_0^{\Sp^1}(\bs{a}_j;b),
\end{align*}
where $g_j = \gcd\{a_i : i\neq j\}$, $\bs{a}_j\in\Z^{n-1}$ is the weight vector formed by removing $a_j$ from
$\bs{a}$, $[r]_s^{-1}$ denotes the multiplicative inverse of $r$ mod $s$,
and $|r|_s$ is the representative of the equivalence class of $r$ mod $s$ such that $1\leq |r|_s \leq s$. When the negative weights in $\bs{a}$ are distinct,
the coefficients can be expressed as
\begin{align}
    \nonumber
    \gamma_0^{\Sp^1}(\bs{a};b)
    &=      \sum\limits_{i=1}^k \frac{-a_i^{n-2}}{\prod\limits_{\substack{j=1 \\ j\neq i}}^n (a_i - a_j) }, \quad\text{and}
    \\
    \label{eq:S1gamma1}
    \gamma_1^{\Sp^1}(\bs{a};b)
    &=      \sum\limits_{i=1}^k \frac{a_i^{n - 3}\left(-2b + \sum\limits_{\substack{j=1 \\ j\neq i}}^n a_j \right)}
                { 2\prod\limits_{\substack{\ell=1 \\ \ell\neq i}}^n (a_i - a_\ell) }
            + \sum\limits_{i=1}^k \sum\limits_{\substack{j=1 \\ j\neq i}}^n \left(\frac{2\big|[a_j]_{g_j}^{-1} b\big|_{g_j} - g_j - 1}{2}\right)
                \frac{-a_i^{n-3}}{ \prod\limits_{\substack{\ell=1 \\ \ell\neq i,j}}^n (a_i - a_\ell) }.
\end{align}
In particular, $\gamma_0^{\Sp^1}(\bs{a};b) = \gamma_0^{\Sp^1}(\bs{a};0)$,
where $\gamma_0^{\Sp^1}(\bs{a};0)$ was computed in \cite[Theorems~6.2 and 6.3]{CowieHerbigSeatonHerden}.
\end{theorem}

\begin{proof}
The hypotheses ensure that the second sum in Equation~\eqref{eq:S1Univar} is empty, so we consider the Laurent expansion
in the first sum. The maximal pole order of a term is $n - 1$ and occurs when $\zeta^{a_j} = 1$ for each
$j$; as $\gcd(a_1,\ldots,a_n) = 1$, this implies $\zeta = 1$. Such a term is of the form
\begin{equation}
\label{eq:S1UnivarLaurentTerm1}
    \frac{t^{b/\alpha_i}}{ \alpha_i
        \prod\limits_{\substack{\ell=1 \\ \ell\neq i}}^n (1 - t^{(c_i - c_\ell)/c_i}) },
\end{equation}
and because the Laurent expansion of $t^{b/\alpha_i}$ begins
\begin{equation}
\label{eq:S1UnivarLaurentSeriesNum}
    t^{b/\alpha_i}
    =   1 - \frac{b}{\alpha_i}(1 - t) + \frac{b(b - \alpha_i)}{2\alpha_i^2}(1 - t)^2 \pm \cdots,
\end{equation}
applying the Cauchy product formula, the degree $1 - n$ Laurent coefficient of the term in
Equation~\eqref{eq:S1UnivarLaurentTerm1} is that of
\[
    \frac{1}{ \alpha_i
        \prod\limits_{\substack{\ell=1 \\ \ell\neq i}}^n (1 - t^{(c_i - c_\ell)/c_i}) }.
\]
In particular, this implies $\gamma_0^{\Sp^1}(\bs{a};b) = \gamma_0^{\Sp^1}(\bs{a};0)$, where the latter was computed in
\cite[Theorem~6.2]{CowieHerbigSeatonHerden}.

The computation of $\gamma_1^{\Sp^1}(\bs{a};b)$ is similar to \cite[Theorem~6.3]{CowieHerbigSeatonHerden}.
We first consider the contribution of terms with $\zeta = 1$ of the form given
in Equation~\eqref{eq:S1UnivarLaurentTerm1}. For simplicity, assume the negative weights of $\bs{a}$ are distinct and
let $c_i = a_i$. Using the Cauchy product formula, the degree $2 - n$ Laurent coefficient of such a term is
\[
    \frac{ a_i^{n-3}\left( -2b +  \sum\limits_{\substack{j=1 \\j\neq i}}^n a_j \right)}
        {2\prod\limits_{\substack{\ell=1 \\ \ell\neq i}}^n (a_i - a_\ell)}.
\]
The only other contributions to $\gamma_1^{\Sp^1}(\bs{a};b)$ are from terms where
$\zeta\neq 1$ is an $\alpha_i$th root of unity such that $\zeta^{\alpha_\ell} = 1$ for all $\ell$ except for one, say
$\zeta^{\alpha_j} \neq 1$. Such a term is then of the form
\[
    \frac{\zeta^b t^{b/\alpha_i}}{ \alpha_i (1 - \zeta^{a_j} t^{(a_i - a_j)/a_i})
        \prod\limits_{\substack{\ell=1 \\ \ell\neq i,j}}^n (1 - t^{(a_i - a_\ell)/a_i}) }
\]
and occurs for each $g_j$th root of unity $\zeta \neq 1$. The degree $2-n$ coefficient of such a term is given by
\[
    \sum\limits_{\substack{\zeta^{g_j} = 1 \\ \zeta\neq 1}}
    \frac{\zeta^b a_i^{n-2}}
    {\alpha_i (1 - \zeta^{a_j}) \prod\limits_{\substack{\ell=1 \\ \ell\neq i,j}}^n (a_i - a_\ell)}
    =
    \frac{- a_i^{n-3}}
    {\prod\limits_{\substack{\ell=1 \\ \ell\neq i,j}}^n (a_i - a_\ell)}
    \sum\limits_{\substack{\zeta^{g_j} = 1 \\ \zeta\neq 1}} \frac{\zeta^b}{1 - \zeta^{a_j}}.
\]
Noting that $g_j$ and $a_j$ are relatively prime by construction, we have
\[
    \sum\limits_{\substack{\zeta^{g_j} = 1 \\ \zeta\neq 1}} \frac{\zeta^b}{1 - \zeta^{a_j}}
    =
    \sum\limits_{\substack{\zeta^{g_j} = 1 \\ \zeta\neq 1}} \frac{\zeta^{[a_j]_{g_j}^{-1} b}}{1 - \zeta}
    =
    \frac{2\big|[a_j]_{g_j}^{-1} b\big|_{g_j} - g_j - 1}{2},
\]
where the last equation follows from \cite[Corollary~3.2]{Gessel}.
Hence the resulting contribution to $\gamma_1^{\Sp^1}(\bs{a}; b)$ is
\[
    \frac{- a_i^{n-3}}
    {\prod\limits_{\substack{\ell=1 \\ \ell\neq i,j}}^n (a_i - a_\ell)}
    \left(\frac{2\big|[a_j]_{g_j}^{-1} b\big|_{g_j} - g_j - 1}{2}\right).
\]
Summing over $i$ and $j$ yields Equation~\eqref{eq:S1gamma1}.

To express $\gamma_1^{\Sp^1}(\bs{a}; b)$ in terms of partial Schur-Laurent polynomials, first note that
\[
    \sum\limits_{i=1}^k \frac{a_i^{u}}
        {\prod\limits_{\substack{\ell=1 \\ \ell\neq i}}^n (a_i - a_\ell)}
    =
    \frac{\bs{sp}_u(\bs{a})}{\prod\limits_{p=1}^k \prod\limits_{q=k+1}^n (a_p - a_q)},
\]
see \cite[Equation~(5.2)]{CowieHerbigSeatonHerden}. We compute
\begin{align*}
    &\sum\limits_{i=1}^k
    \frac{ a_i^{n-3}\left( -2b +  \sum\limits_{\substack{j=1 \\ j\neq i}}^n a_j \right)}
        {2\prod\limits_{\substack{\ell=1 \\ \ell\neq i}}^n (a_i - a_\ell)}
    +
    \sum\limits_{i=1}^k \sum\limits_{\substack{ j=1 \\ j\neq i}}^n
    \left(\frac{2\big|[a_j]_{g_j}^{-1} b\big|_{g_j} - g_j - 1}{2}\right)
    \frac{- a_i^{n-3}}
        {\prod\limits_{\substack{\ell=1 \\ \ell\neq i,j}}^n (a_i - a_\ell)}
    \\&=
    \sum\limits_{i=1}^k
    \frac{ -2b a_i^{n-3} - a_i^{n-2} + a_i^{n-3} \bs{e}_1(\bs{a}) }
        {2\prod\limits_{\substack{\ell=1 \\ \ell\neq i}}^n (a_i - a_\ell)}
    +
    \sum\limits_{j=1}^n
    \left(\frac{2\big|[a_j]_{g_j}^{-1} b\big|_{g_j} - g_j - 1}{2}\right)
    \gamma_0^{\Sp^1}(\bs{a}_j; b)
    \\&=
    \sum\limits_{i=1}^k
    \frac{ \Big(\bs{e}_1(\bs{a}) - 2b\Big) \bs{sp}_{n-3}(\bs{a}) - \bs{sp}_{n-2}(\bs{a}) }
        {2\prod\limits_{p=1}^k \prod\limits_{q=k+1}^n (a_p - a_q)}
    +
    \sum\limits_{j=1}^n
    \left(\frac{2\big|[a_j]_{g_j}^{-1} b\big|_{g_j} - g_j - 1}{2}\right)
    \gamma_0^{\Sp^1}(\bs{a}_j; b).
    \qedhere
\end{align*}
\end{proof}

We now consider the case of a cotangent-lifted representation and compute $\gamma_0^{\Sp^1}(\bs{a}^c;b)$
and $\gamma_1^{\Sp^1}(\bs{a}^c;b)$ where $\gamma_0^{\Sp^1}(\bs{a}^c;b)$ occurs in degree $2n - 1$.
In this case, we do not need to restrict $n$ nor the signs of the weights. We continue to assume without
loss of generality that $b\geq 0$ by Remark~\ref{rem:S1MaxSEmpty}.

\begin{theorem}[$\gamma_0$ and $\gamma_1$ for $\Sp^1$-covariants of a cotangent-lifted representation]
\label{thrm:S1LaurentCotan}
Let $V\simeq \C^n$ be a faithful representation of $\Sp^1$ with weight vector
$\bs{a} = (-\alpha_1,\ldots,-\alpha_k,\alpha_{k+1},\ldots,\alpha_n)$ where each $\alpha_i > 0$,
let $\bs{a}^c = (\bs{a}, -\bs{a})$ denote the weight vector of the cotangent lift of $V$,
let $\bs{\alpha} = (\alpha_1,\ldots,\alpha_n)$, and let $W\simeq \C$ be the irreducible representation
of the circle with weight $-b$ where $b\geq 0$. Then the first two Laurent coefficients $\gamma_0^{\Sp^1}(\bs{a}^c;b)$
and $\gamma_1^{\Sp^1}(\bs{a}^c;b)$ of the Hilbert series $\Hilb_{\bs{a}^c;b}^{\Sp^1}(t)$ of covariants
$\HOM(V\oplus V^\ast,W)^{\Sp^1} = (\C[V\oplus V^\ast]\otimes W)^{\Sp^1}$ are given by
\begin{align}
    \label{eq:S1LaurentCotanGamma0}
    \gamma_0^{\Sp^1}(\bs{a}^c;b)
    &=      \frac{ \bs{s}_{(n-2,n-2,n-3,\ldots,1,0)}(\bs{\alpha}) }
                { 2\bs{s}_{(n-1,n-2,n-3,\ldots,1,0)}(\bs{\alpha})}, \quad\text{and}
    \\
    \label{eq:S1LaurentCotanGamma1}
    \gamma_1^{\Sp^1}(\bs{a}^c;b)
    &=      \frac{\gamma_0^{\Sp^1}(\bs{a}^c;b)}{2}.
\end{align}
When the $\alpha_i$ are distinct, $\gamma_0^{\Sp^1}(\bs{a}^c;b)$ can be expressed as
\[
    \gamma_0^{\Sp^1}(\bs{a}^c;b)
    =      \sum\limits_{i=1}^n \frac{\alpha_i^{2n-3}}
                {2\prod\limits_{\substack{j=1 \\ j\neq i}}^n (\alpha_i^2 - \alpha_j^2) }.
\]
In particular, $\gamma_m^{\Sp^1}(\bs{a}^c;b) = \gamma_m^{\Sp^1}(\bs{a}^c;0)$ for $m=1,2$,
where the $\gamma_m^{\Sp^1}(\bs{a}^c;0)$ were computed in \cite[Theorem~5.1 and Remark~5.3]{HerbigSeaton}.
\end{theorem}
\begin{proof}
We consider the Laurent expansion of Equation~\eqref{eq:S1CotanUnivarDegen}, where our hypotheses imply that
the second sum vanishes.
%Because the univariate Hilbert series is invariant under permutations of the weights,
%we assume without loss of generality that each $a_i$, and hence each $c_i$, is positive so that $a_i = \alpha_i$.
The maximum pole order of a term
\[
        \frac{\zeta^{a_i b/\alpha_i} t^{b/\alpha_i}}{ (1 - t^2) \alpha_i
        \prod\limits_{\substack{j=1 \\ j\neq i}}^n
            (1 - \zeta^{-a_j} t^{(c_i - c_j)/c_i} )
            (1 - \zeta^{a_j} t^{(c_i + c_j)/c_i} ) }.
\]
is $2n - 1$ and occurs when each $\zeta^{a_j} = 1$; this implies by the faithfulness of $V$ that $\zeta = 1$. Note that
if $\zeta\neq 1$, then the pole order is at most $2n - 3$, so the only terms that contribute to $\gamma_0$ and $\gamma_1$
are of the form
\begin{equation}
\label{eq:S1LaurentCotanTerm}
        \frac{t^{b/ \alpha_i}}{ (1 - t^2) \alpha_i
        \prod\limits_{\substack{j=1 \\ j\neq i}}^n
            (1 - t^{(c_i - c_j)/c_i} )
            (1 - t^{(c_i + c_j)/c_i} ) }.
\end{equation}
Applying the Cauchy product formula to such a term and summing over $i$, we obtain
\begin{align*}
    \gamma_0^{\Sp^1}(\bs{a}^c;b)
    =       \lim\limits_{(c_1,\ldots,c_n)\to\bs{a}}
            \sum\limits_{i=1}^n \frac{c_i^{2n-2}}
                { 2\alpha_i \prod\limits_{\substack{j = 1 \\ j \neq i}}^n
                    (c_i^2 - c_j^2)}
    =       \lim\limits_{(c_1,\ldots,c_n)\to\bs{a}}
            \sum\limits_{i=1}^n \frac{\lvert c_i\rvert^{2n-3}}
                { 2\prod\limits_{\substack{j = 1 \\ j \neq i}}^n
                    (c_i^2 - c_j^2)},
\end{align*}
which is shown to be equal to the expression in Equation~\eqref{eq:S1LaurentCotanGamma0} in
\cite[Section~5.2]{HerbigSeaton}.

To compute $\gamma_1^{\Sp^1}(\bs{a}^c;b)$, we apply the Cauchy product formula to the term in
Equation~\eqref{eq:S1LaurentCotanTerm} for each $i$, yielding
\begin{multline*}
 \gamma_1^{\Sp^1}(\bs{a}^c;b)=   \lim\limits_{(c_1,\ldots,c_n)\to\bs{a}}
    \sum\limits_{i=1}^n \left[ \sum\limits_{\substack{j=1 \\ j\neq i}}^n \frac{c_i^{2n-4}}
        {2\alpha_i \prod\limits_{\substack{\ell = 1 \\ \ell \neq i,j}}^n
            (c_i^2 - c_\ell^2)}
        \left( \frac{c_i c_j - c_i c_j}{2(c_i^2 - c_j^2) }
        \right)
        + \frac{c_i^{2n-2}}
                { 4\alpha_i\prod\limits_{\substack{j = 1 \\ j \neq i}}^n
                    (c_i^2 - c_j^2)}
        - \frac{b c_i^{2n-2}}
                {2\alpha_i^2 \prod\limits_{\substack{j = 1 \\ j \neq i}}^n
                    (c_i^2 - c_j^2)}
        \right]
    \\
    = \lim\limits_{(c_1,\ldots,c_n)\to\bs{a}}
    \left[ \sum\limits_{i=1}^n \frac{\lvert c_i\rvert^{2n-3}}
                { 4\prod\limits_{\substack{j = 1 \\ j \neq i}}^n
                    (c_i^2 - c_j^2)}
        - \frac b2 \sum\limits_{i=1}^n \frac{c_i^{2n-4}}
                {\prod\limits_{\substack{j = 1 \\ j \neq i}}^n
                    (c_i^2 - c_j^2)}\right].
\end{multline*}
The first of the resulting sums is equal to $\gamma_0^{\Sp^1}(\bs{a}^c;b)/2$. To see that the second sum
vanishes, rewrite
\[
    \sum\limits_{i=1}^n \frac{ c_i^{2n-4}}
        {\prod\limits_{\substack{j = 1 \\ j \neq i}}^n
            (c_i^2 - c_j^2)}
    =
    \frac{ \sum\limits_{i=1}^n (-1)^{i+1} c_i^{2n-4}
        \prod\limits_{\substack{1\leq j < k \leq n \\ j,k\neq i}} (c_j^2 - c_k^2)}
        { \prod\limits_{1\leq j < k \leq n} (c_j^2 - c_k^2)}
\]
and observe that the numerator is the cofactor expansion along the first row of a matrix whose first two rows
are identical, both listing powers $c_i^{2n-4}$.
\end{proof}

% xxxxxxxxxxxxxxxxxxxxxxxxxxxxxxxxxxxxxxxxxxxxxxxxxxxxxxxxxxxxxxxxxxxxxxxxx

\subsection{Laurent coefficients of multigraded Hilbert series}
\label{subsec:S1LaurentMultigrad}

Using the approach of Section~\ref{subsec:S1Laurent}, we can also consider the Laurent coefficients of the multigraded
Hilbert series. There are several approaches to Laurent expansions of mutivariate functions, see
\cite{AparicioMonforteKauers} for a careful exposition, and the meaning of the coefficients is in this context less clear.
Hence, we give some sample calculations of the first few iterated Laurent series coefficients of
$\Hilb_{\bs{a}}^{\Sp^1,\on}(s,t)$ computed in Corollary~\ref{cor:S1SymplecticBivar}.
For simplicity, we assume that $n \geq 3$ and the $a_i$ are distinct and hence consider Equation~\eqref{eq:S1SymplecticBivarGeneric};
we furthermore assume that the weights are pairwise relatively prime, i.e., $\gcd(\alpha_i,\alpha_j) = 1$ for $i\neq j$,
which in particular implies that $V$ is faithful. We let $\gamma_{i,j}^{\Sp^1, \on; s,t}(\bs{a})$ denote the $j$th coefficient of the expansion at $t = 1$
of the $i$th coefficient of the expansion at $s = 1$, where $i = 0$ and $j = 0$ indicate the first nonzero coefficient of
the corresponding expansion.
Note that $\Hilb_{\bs{a}}^{\Sp^1,\on}(s,t) = \Hilb_{\bs{a}}^{\Sp^1,\on}(t,s)$ by Remark \ref{rem:S1BivarSEmpty} so that in this instance,
$\gamma_{i,j}^{\Sp^1, \on; s,t}(\bs{a}) =\gamma_{i,j}^{\Sp^1, \on; t,s}(\bs{a})$ for all $i,j$. However, changing the expansion order
yields alternate formulas for these coefficients.

We first expand in $s = 1$ and then $t = 1$. Choose $i \leq k$, and hence consider a term in the first sum of
Equation~\eqref{eq:S1SymplecticBivarGeneric}. Such a term has a pole of order at most $n - 1$, with this pole
order obtained if and only if $\zeta^{a_j} = 1$ for each $j$. As $V$ is faithful, this implies that $\zeta = 1$.
Note that if $\zeta\neq 1$, then there is no pole at $s = 1$ so that, as $n \geq 3$, the corresponding term will
not contribute to the first two Laurent coefficients at $s = 1$.

Hence we consider a term of the form
\[
    \frac{1}{ \alpha_i
        \prod\limits_{\substack{j=1 \\ j\neq i}}^n
            (1 - s^{(a_i - a_j)/a_i} )
            (1 - s^{a_j/a_i} t) }.
\]
By the Cauchy product formula, the expansion at $s = 1$ of such a term begins
\begin{align*}
    &\frac{a_i^{n-1}}
        { \alpha_i (1 - t)^{n - 1} \prod\limits_{\substack{j=1 \\ j\neq i}}^n (a_i - a_j)}
    (1 - s)^{1 - n}
    \\&\qquad +
    \left(
        \sum\limits_{\substack{j=1\\ j\neq i}}^n \frac{ - a_i^{n - 2}a_j}
            { 2 \alpha_i (a_i - a_j) (1 - t)^{n-1} \prod\limits_{\substack{\ell=1 \\ \ell\neq i,j}}^n (a_i - a_\ell)}
        +
        \sum\limits_{\substack{j=1\\ j\neq i}}^n \frac{ - a_i^{n-2} a_j t}
            {\alpha_i (1 - t)^{n} \prod\limits_{\substack{\ell=1 \\ \ell\neq i}}^n (a_i - a_\ell) }
    \right)(1 - s)^{2 - n}+\cdots
    \\&=
    \frac{ - a_i^{n-2}}
        { (1 - t)^{n - 1} \prod\limits_{\substack{j=1 \\ j\neq i}}^n (a_i - a_j)}
    (1 - s)^{1 - n}
    +
    \sum\limits_{\substack{j=1\\ j\neq i}}^n \frac{ a_i^{n - 3}a_j(1 + t) }
        { 2 (1 - t)^n \prod\limits_{\substack{\ell=1 \\ \ell\neq i}}^n (a_i - a_\ell)}
    (1 - s)^{2 - n}+\cdots.
\end{align*}
Summing over the corresponding $i$ and expanding at $t = 1$ yields for the first term
\[
    (1 - s)^{1 - n}(1 - t)^{1 - n}
    \sum\limits_{i=1}^k \frac{ - a_i^{n-2}}
        { \prod\limits_{\substack{j=1 \\ j\neq i}}^n (a_i - a_j)},
\]
and for the second term
\[
    (1 - s)^{2 - n} (1 - t)^{-n}
    \sum\limits_{i=1}^k
    \sum\limits_{\substack{j=1\\ j\neq i}}^n \frac{ a_i^{n - 3}a_j }
        {\prod\limits_{\substack{\ell=1 \\ \ell\neq i}}^n (a_i - a_\ell)}
    +
    (1 - s)^{2 - n} (1 - t)^{1-n}
    \sum\limits_{i=1}^k
    \sum\limits_{\substack{j=1\\ j\neq i}}^n \frac{ - a_i^{n - 3}a_j }
        { 2 \prod\limits_{\substack{\ell=1 \\ \ell\neq i}}^n (a_i - a_\ell)}.
\]
When $i > k$, the corresponding term is of the form
\[
    \frac{1}{ \alpha_i
        \prod\limits_{\substack{j=1 \\ j\neq i}}^n
            (1 - \zeta^{a_j} s t^{a_j/a_i} )
            (1 - \zeta^{-a_j} t^{(a_i - a_j)/a_i} ) }.
\]
Such a term does not have a pole at $s = 1$ regardless of the value of $\zeta$ so, as $n\geq 3$, will not contribute to the first two Laurent coefficients at $s = 1$.
Hence, we have
\[
    \gamma_{0,0}^{\Sp^1, \on; s,t}(\bs{a})
        =   \sum\limits_{i=1}^k \frac{ - a_i^{n-2}}
                { \prod\limits_{\substack{j=1 \\ j\neq i}}^n (a_i - a_j)},
\]
which is equal to $\gamma_0^{\Sp^1}(\bs{a};0)$; see Theorem \ref{thrm:S1Laurent}. Similarly,
\[
    \gamma_{1,0}^{\Sp^1, \on; s,t}(\bs{a})
    =   \sum\limits_{i=1}^k
        \sum\limits_{\substack{j=1\\ j\neq i}}^n \frac{ a_i^{n - 3}a_j }
            {\prod\limits_{\substack{\ell=1 \\ \ell\neq i}}^n (a_i - a_\ell)}
    \quad\text{ and }\quad
    \gamma_{1,1}^{\Sp^1, \on; s,t}(\bs{a})
    =   \sum\limits_{i=1}^k
        \sum\limits_{\substack{j=1\\ j\neq i}}^n \frac{ - a_i^{n - 3}a_j }
            { 2 \prod\limits_{\substack{\ell=1 \\ \ell\neq i}}^n (a_i - a_\ell)}.
\]

Expanding at $t = 1$ and then $s = 1$ is similar.
If $i \leq k$, then we have a term of the form
\[
    \frac{1}{ \alpha_i
        \prod\limits_{\substack{j=1 \\ j\neq i}}^n
            (1 - \zeta^{a_j} s^{(a_i - a_j)/a_i} )
            (1 - \zeta^{-a_j} s^{a_j/a_i} t) }
\]
which does not have a pole at $t=1$ and hence, as $n\geq 3$, will not contribute to the first two Laurent coefficients at $t=1$.
If $i \geq k + 1$, we again have a pole at $t = 1$ only if $\zeta = 1$, in which case we have a term of the form
\[
    \frac{1}{ \alpha_i
        \prod\limits_{\substack{j=1 \\ j\neq i}}^n
            (1 - s t^{a_j/a_i} )
            (1 - t^{(a_i - a_j)/a_i} ) }.
\]
This term has a pole of order $n - 1$ at $t = 1$, and the expansion at $t = 1$ begins
\begin{align*}
    &\frac{ a_i^{n-1} }
        {\alpha_i (1 - s)^{n-1} \prod\limits_{\substack{j=1 \\ j\neq i}}^n (a_i - a_j)}
        (1 - t)^{1-n}
    \\& \qquad + \left(
    \sum\limits_{\substack{j=1 \\ j\neq i}}^n
        \frac{ - a_i^{n-2} a_j}
        {2 \alpha_i (1 - s)^{n-1} (a_i - a_j) \prod\limits_{\substack{\ell=1 \\ \ell \neq i,j }}^n (a_i - a_\ell) }
    + \sum\limits_{\substack{j=1 \\ j\neq i}}^n
        \frac{ - a_i^{n-2} a_j s}
        {\alpha_i (1 - s)^{n} \prod\limits_{\substack{\ell=1 \\ \ell \neq i }}^n (a_i - a_\ell)}
    \right)(1 - t)^{2 - n}+\cdots
    \\& =
    \frac{ a_i^{n-2} }
        {(1 - s)^{n-1} \prod\limits_{\substack{j=1 \\ j\neq i}}^n (a_i - a_j)}
        (1 - t)^{1-n}
    +
    \sum\limits_{\substack{j=1 \\ j\neq i}}^n
        \frac{ - a_i^{n-3} a_j (1 + s)}
        {2 (1 - s)^n \prod\limits_{\substack{\ell=1 \\ \ell \neq i }}^n (a_i - a_\ell) }
    (1 - t)^{2 - n}+\cdots.
\end{align*}
Summing over $i \geq k + 1$ and expanding at $s = 1$, we obtain
\[
    \gamma_{0,0}^{\Sp^1, \on; t,s}(\bs{a})
    =   \sum\limits_{i=k+1}^n
        \frac{ a_i^{n-2} }
            { \prod\limits_{\substack{j=1 \\ j\neq i}}^n (a_i - a_j) },
\]
\[
    \gamma_{1,0}^{\Sp^1, \on; t,s}(\bs{a})
    =   \sum\limits_{i=k+1}^n
        \sum\limits_{\substack{j=1 \\ j\neq i}}^n
        \frac{ - a_i^{n-3} a_j}
            {\prod\limits_{\substack{\ell=1 \\ \ell \neq i }}^n (a_i - a_\ell) },
    \quad\text{ and }\quad
    \gamma_{1,1}^{\Sp^1, \on; t,s}(\bs{a})
    =   \sum\limits_{i=k+1}^n
        \sum\limits_{\substack{j=1 \\ j\neq i}}^n
        \frac{ a_i^{n-3} a_j}
            {2 \prod\limits_{\substack{\ell=1 \\ \ell \neq i }}^n (a_i - a_\ell) }.
\]

%which is equal to $\gamma_0^{\Sp^1}(\bs{a};0)$ and hence $\gamma_{0,0}^{\Sp^1, \on; s,t}(\bs{a})$,
%which can be seen by applying Theorem \ref{thrm:S1Laurent} to $-\bs{a}$, the weight vector
%of a representation with the same invariants as those of $\bs{a}$. As well,
%\[
%    \gamma_{1,0}^{\Sp^1, \on; t,s}(\bs{a})
%    =   \sum\limits_{i=k+1}^n
%        \sum\limits_{\substack{j=1 \\ j\neq i}}^n
%        \frac{ - a_i^{n-3} a_j}
%            {\prod\limits_{\substack{\ell=1 \\ \ell \neq i }}^n (a_i - a_\ell) },
%    \quad\text{ and }\quad
%    \gamma_{1,1}^{\Sp^1, \on; t,s}(\bs{a})
%    =   \sum\limits_{i=k+1}^n
%        \sum\limits_{\substack{j=1 \\ j\neq i}}^n
%        \frac{ a_i^{n-3} a_j}
%            {2 \prod\limits_{\substack{\ell=1 \\ \ell \neq i }}^n (a_i - a_\ell) }.
%\]

% xxxxxxxxxxxxxxxxxxxxxxxxxxxxxxxxxxxxxxxxxxxxxxxxxxxxxxxxxxxxxxxxxxxxxxxxx
% xxxxxxxxxxxxxxxxxxxxxxxxxxxxxxxxxxxxxxxxxxxxxxxxxxxxxxxxxxxxxxxxxxxxxxxxx
% xxxxxxxxxxxxxxxxxxxxxxxxxxxxxxxxxxxxxxxxxxxxxxxxxxxxxxxxxxxxxxxxxxxxxxxxx

\section{Invariants and covariants of $\OO_2(\R)$}
\label{sec:O2}

In this section, we consider the Hilbert series of invariants and covariants of representations of $\OO_2(\R)$.
We consider $\OO_2(\R)$ as the set of $2\times 2$ real matrices of the form
\[
    r_\theta    =   \begin{pmatrix} \cos\theta & -\sin\theta \\ \sin\theta & \cos\theta \end{pmatrix}
    \quad\text{and}\quad
    s_\theta    =    \begin{pmatrix} \cos \theta & \sin \theta \\ \sin \theta & -\cos \theta \end{pmatrix}.
\]

% xxxxxxxxxxxxxxxxxxxxxxxxxxxxxxxxxxxxxxxxxxxxxxxxxxxxxxxxxxxxxxxxxxxxxxxxx

\subsection{The irreducible representations of $\OO_2(\R)$}
\label{subsec:O2Irreps}

Here, we briefly recall the representation theory of $\OO_2(\R)$ and refer the reader to
\cite[Theorem~7.2.1]{PalmThesis} or \cite[Section~11.2]{Knightly} for more details.

For each $a\in \Z$, let $\tau_a\co\OO_2(\R)\to\U_2$ denote the representation given by
\begin{align*}
    \tau_a  \co&
        \begin{pmatrix} \cos\theta & -\sin\theta \\ \sin\theta & \cos\theta \end{pmatrix}
        \mapsto
        \begin{pmatrix} e^{a\theta\sqrt{-1}} & 0 \\ 0 & e^{-a\theta\sqrt{-1}} \end{pmatrix}
    \\
        &\begin{pmatrix} \cos \theta & \sin \theta \\ \sin \theta & -\cos \theta \end{pmatrix}
        \mapsto
        \begin{pmatrix} 0 & e^{a\theta\sqrt{-1}} \\ e^{-a\theta\sqrt{-1}} & 0 \end{pmatrix}
\end{align*}
Identifying $\Sp^1$ with $\SO_2(\R)\leq\OO_2(\R)$ as above $\tau_a$ is the representation of $\OO_2(\R)$
induced by $\epsilon_a$. The representation $\tau_0$ splits into a $1$-dimensional trivial representation
and the $1$-dimensional representation $\det$, and for $a > 0$, $\tau_a$ is equivalent to $\tau_{-a}$.
Hence, every non-trivial finite-dimensional irreducible unitary representation of $\OO_2(\R)$ is isomorphic to
$\det$ or $\tau_a$ for an integer $a > 0$. Note that $\det$ has kernel $\SO_2(\R)$, and the kernel of $\tau_a$ for
$a \geq 1$ is the set of $a$th roots of unity in $\Sp^1\simeq\SO_2(\R)$. In the latter case, the quotient of
$\OO_2(\R)$ by the finite group of $a$th roots of unity is isomorphic to $\OO_2(\R)$, and the resulting
faithful representation is $\tau_1$.

% xxxxxxxxxxxxxxxxxxxxxxxxxxxxxxxxxxxxxxxxxxxxxxxxxxxxxxxxxxxxxxxxxxxxxxxxx

\subsection{The Hilbert series of covariants of $\OO_2$}
\label{subsec:O2Hilb}

Let $V$ be a finite-dimensional unitary representation of $\OO_2(\R)$. As in Remark \ref{rem:S1Faithful}, we assume that $V$
has no trivial subrepresentations. Then there are integers $\alpha_1,\ldots,\alpha_n,d$ with $n \geq 0$ and each
$\alpha_i > 0$ such that $V$ is of the form
\[
    V   =   \left(\bigoplus\limits_{i=1}^n \tau_{\alpha_i}\right)\oplus d\det.
\]
We will refer to this representation as $V_{\bs{\alpha},d}$ where $\bs{\alpha} = (\alpha_1,\ldots,\alpha_n)$. Note that
the corresponding weight vector for the action of the torus $\SO_2(\R)\leq\OO_2(\R)$ is $\bs{\alpha}^c$ concatenated with
$d$ zero entries. If $n = 0$, then $V$ is simply copies of $\det$ and $\SO_2(\R)$ acts trivially. If $n\ge 1$,
then the kernel of the action is the set of $\gcd(\alpha_1,\ldots,\alpha_n)$th roots of unity. As in Remark \ref{rem:S1Faithful},
we will assume that $V$ is faithful, i.e., $n\ge 1$ and $\gcd(\alpha_1,\ldots,\alpha_n)=1$.

\begin{theorem}[Maximally graded Hilbert series of $\OO_2$-covariants]
\label{thrm:O2Max}
Let $\bs{\alpha} = (\alpha_1,\ldots,\alpha_n)$ where $n\geq 1$, each $\alpha_i > 0$, and $d\geq 0$; let
$V_{\bs{\alpha},d}\simeq \C^{2n+d}$ be the corresponding faithful representation of $\OO_2(\R)$;
and let $W$ be an irreducible representation of $\OO_2(\R)$ with character $\chi = \chi_W$.
Then either $W = \tau_\beta$ for some $\beta>0$ or $W$ is the trivial representation or $\det$, where in the latter two cases we set $\beta = 0$.
Let $\bs{t} = (t_1,\ldots,t_{n+d})$, and let
$\Hilb_{\bs{\alpha},d;W}^{\OO_2}(\bs{t})$ denote the Hilbert series for the maximal $\N^n$-grading of the
covariants $\HOM(V,W)^{\OO_2(\R)} = (\C[V]\otimes W)^{\OO_2(\R)}$. Then
\begin{multline}
\label{eq:O2Max}
    \Hilb_{\bs{\alpha},d;W}^{\OO_2}(\bs{t})
    =
    \frac{C_1}{2 \prod\limits_{j=1}^d (1 - t_{n+j})}
     \sum_{\substack{i=1 \\ \zeta^{\alpha_i} = 1}}^n \frac{\zeta^\beta t_i^{\beta/\alpha_i}}{ \alpha_i
        (1 -  t_i^2) \prod\limits_{\substack{j=1 \\ j\neq i}}^n
            (1 - \zeta^{-\alpha_j} t_j t_i^{- \alpha_j/\alpha_i})
            (1 - \zeta^{\alpha_j} t_j t_i^{\alpha_j/\alpha_i} ) }
    \\
    + \frac{C_2}{2 \prod\limits_{j=1}^n (1 - t_j^2) \prod\limits_{j=1}^d (1 + t_{n+j})},
\end{multline}
where
\[
    C_1 =   \begin{cases}
                1,  &  \text{if } W = \det \text{ or the trivial representation},
                \\
                2,  &  \text{if } W = \tau_\beta \text{ for } \beta > 0,
            \end{cases}
\]
and
\[
    C_2 =   \begin{cases}
                -1, &  \text{if } W = \det,
                \\
                0,  &  \text{if } W = \tau_\beta \text{ for } \beta > 0,
                \\
                1,  &  \text{if } W \text{ is the trivial representation}.
            \end{cases}
\]
\end{theorem}
\begin{proof}
Because $\OO_2(\R)$ has two connected components, the Molien-Weyl Theorem yields an integral over each connected
component with a prefactor of $1/2$ for each; see \cite[Section~4.1]{HananyMekareeya}. We first consider the case of a rotation
$z\in\OO_2(\R)^\circ = \SO_2(\R)$, which we consider as an element of $\Sp^1\subset\C$ as above.
If $W = \tau_\beta$ for some $\beta > 0$, then $\chi(z) = z^\beta + z^{-\beta}$; if $W$ is trivial or $\det$,
then $\chi(z) = 1$ is constant. Hence, the integral (with prefactor~$1/2$) over $\SO_2(\R)$ is given by
\begin{multline}
\label{eq:O2MaxIntegral1}
    \frac{1}{4\pi\sqrt{-1}} \int_{\Sp^1} \frac{\chi(z)\, dz}{z
        \prod\limits_{j=1}^n (1 - t_jz^{\alpha_j})(1 - t_jz^{-\alpha_j})
        \prod\limits_{j=1}^d (1 - t_{n+j}) }
    \\=
    \frac{1}{2 \prod\limits_{j=1}^d (1 - t_{n+j})}\left( \frac{1}{2\pi\sqrt{-1}}
        \int_{\Sp^1} \frac{\chi(z) \, dz}{z \prod\limits_{j=1}^n (1 - t_jz^{\alpha_j})(1 - t_jz^{-\alpha_j}) }\right).
\end{multline}
On the connected component of reflections, $\chi(z)$ is constant; specifically, $\chi(z) = 0$ if $W = \tau_\beta$
for some $\beta > 0$, $\chi(z) = -1$ if $W = \det$, and $\chi(z) = 1$ if $W$ is the trivial representation. Hence, denoting the constant
$\chi(z)$ simply as $\chi$, the integral (with prefactor~$1/2$) is
\begin{align}
    \nonumber
    \frac{1}{4\pi\sqrt{-1}} \int_{\Sp^1} \frac{\chi \, dz}{z
        \prod\limits_{j=1}^n (1 - t_j^2) \prod\limits_{j=1}^d (1 + t_{n+j}) }
    &=
    \label{eq:O2MaxIntegral2}
    \frac{\chi}{4\pi\sqrt{-1} \prod\limits_{j=1}^n (1 - t_j^2) \prod\limits_{j=1}^d (1 + t_{n+j})}
        \int_{\Sp^1} \frac{dz}{z}
    \\&=
    \frac{\chi   }{2 \prod\limits_{j=1}^n (1 - t_j^2) \prod\limits_{j=1}^d (1 + t_{n+j})}.
\end{align}

The integral in the second line of Equation~\eqref{eq:O2MaxIntegral1} can be interpreted as a specialization of that of Theorem~\ref{thrm:S1Max}
with the weight vector $\bs{a} = (-\alpha_1,\ldots,-\alpha_n,\alpha_1,\ldots,\alpha_n)$ and the substitution $t_{n+i} := t_i$.
When $W = \tau_\beta$, we apply the theorem with $b = \pm \beta$ and sum the results; when $W$ is the trivial representation
or $\det$, we use $b = 0$.

When $b \geq 0$, the assumption that each $\alpha_i > 0$ implies
$b + \sum_{i=1}^n \alpha_i > 0$ so that the integral (along with the prefactor in parentheses) is
\begin{equation}
\label{eq:O2MaxIntegral1Positive}
    \sum_{\substack{i=1 \\ \zeta^{\alpha_i} = 1}}^n \frac{\zeta^b t_i^{b/\alpha_i}}{ \alpha_i
        (1 -  t_i^2) \prod\limits_{\substack{j=1 \\ j\neq i}}^n
            (1 - \zeta^{-\alpha_j} t_j t_i^{- \alpha_j/\alpha_i})
            (1 - \zeta^{\alpha_j} t_j t_i^{\alpha_j/\alpha_i} ) }.
\end{equation}
When $b < 0$, we can use the fact that $-\bs{a} = \bs{a}$, up to permuting the weights; see Remarks~\ref{rem:S1MaxSEmpty}
and \ref{rem:S1BivarSEmpty}. Then as the permutation of weights corresponds to transposing $t_i$ with $t_{n+i}$ for
$1\leq i\leq n$, after applying the substitution $t_{n+i} := t_i$, we obtain Equation~\eqref{eq:O2MaxIntegral1Positive}
again. Combining Equations~\eqref{eq:O2MaxIntegral1}, \eqref{eq:O2MaxIntegral2}, and \eqref{eq:O2MaxIntegral1Positive}
completes the proof.
\end{proof}

Substituting $t_i = t$ in $\Hilb_{\bs{\alpha},d;W}^{\OO_2}(\bs{t})$ and, in the case that the $\alpha_i$ are not distinct,
applying the analytic continuation argument of \cite[Section~3.3]{HerbigHerdenSeatonT2}, we obtain the univariate Hilbert
series for the covariants of a unitary representation of $\OO_2(\R)$.

\begin{corollary}[Univariate Hilbert series of $\OO_2$-covariants]
\label{cor:O2Univar}
Let $\bs{\alpha} = (\alpha_1,\ldots,\alpha_n)$ where $n\geq 1$, each $\alpha_i > 0$, and $d\geq 0$, let
$V_{\bs{\alpha},d}\simeq \C^{2n+d}$ be the corresponding faithful representation of $\OO_2(\R)$, and let $W$ be an irreducible representation of $\OO_2(\R)$.
The univariate Hilbert series $\Hilb_{\bs{\alpha},d;W}^{\OO_2}(t)$ of the algebra of covariants
$\HOM(V,W)^{\OO_2(\R)} = (\C[V]\otimes W)^{\OO_2(\R)}$ is given by
\begin{equation}
\label{eq:O2Univar}
    \lim\limits_{(c_1,\ldots,c_n)\to\bs{\alpha}}\frac{C_1}{2(1 - t)^d}\!\!
     \sum_{\substack{i=1 \\ \zeta^{\alpha_i} = 1}}^n \frac{\zeta^\beta t^{\beta/\alpha_i}}{ \alpha_i
        (1 -  t^2) \prod\limits_{\substack{j=1 \\ j\neq i}}^n
            (1 - \zeta^{-\alpha_j} t^{(c_i - c_j)/c_i})
            (1 - \zeta^{\alpha_j} t^{(\alpha_i + \alpha_j)/\alpha_i} ) }
    + \frac{C_2}{2(1 - t^2)^n (1 + t)^d},
\end{equation}
where $\beta$, $C_1$, and $C_2$ are as in Theorem~\ref{thrm:O2Max}.
If the $\alpha_i$ are distinct, then Equation~\eqref{eq:O2Univar} simplifies to
\[
    \frac{C_1}{2(1 - t)^d}
     \sum_{\substack{i=1 \\ \zeta^{\alpha_i} = 1}}^n \frac{\zeta^\beta t^{\beta/\alpha_i}}{ \alpha_i
        (1 -  t^2) \prod\limits_{\substack{j=1 \\ j\neq i}}^n
            (1 - \zeta^{-\alpha_j} t^{(\alpha_i - \alpha_j)/\alpha_i})
            (1 - \zeta^{\alpha_j} t^{(\alpha_i + \alpha_j)/\alpha_i} ) }
    + \frac{C_2}{2(1 - t^2)^n (1 + t)^d}.
\]
\end{corollary}

\begin{remark}
\label{rem:O2MaxAsS1}
Note that we can express $\Hilb_{\bs{\alpha},d;W}^{\OO_2}$ in terms of Hilbert series associated to the
restricted $\Sp^1$-representation as follows. That is, Equation~\eqref{eq:O2Max} can be written
\begin{equation}
\label{eq:O2MaxW1dimAsS1}
    \Hilb_{\bs{\alpha},d;W}^{\OO_2}(\bs{t})
    =
    \frac{C_1}{2 \prod\limits_{j=1}^d (1 - t_{n+j})}
    \Hilb_{(-\bs{\alpha},\bs{\alpha});\beta}^{\Sp^1}(t_1,\ldots,t_n,t_1,\ldots,t_n)
    + \frac{C_2}{2 \prod\limits_{j=1}^n (1 - t_j^2) \prod\limits_{j=1}^d (1 + t_{n+j})},
\end{equation}
where $\beta = 0$ when $W$ is $\det$ or the trivial representation. In particular, let us temporarily relax the hypothesis that $V$ is faithful
for this remark. If $V$ contains at least
one $\tau_{\alpha_i}$ summand, $W = \tau_\beta$, and $g:=\gcd(\alpha_1,\ldots,\alpha_n) > 1$ does not divide $\beta$,
then $\Hilb_{(-\bs{\alpha},\bs{\alpha});\beta}^{\Sp^1}(t_1,\ldots,t_n,t_1,\ldots,t_n) = C_2 = 0$
so that the Hilbert series is $0$ and there are no covariants; see Remark~\ref{rem:S1Faithful}.
If $g$ divides $\beta$, we can without loss of generality take the quotient of $\OO_2$ by the kernel of the action
on $V$, resulting in the faithful representation $V_{\bs{\alpha}/g,d}$ and $W = \tau_{\beta/g}$.

Similarly, Equation~\eqref{eq:O2Univar} can be written
\begin{equation}
\label{eq:O2UnivarW1dimAsS1}
    \Hilb_{\bs{\alpha},d;W}^{\OO_2}(t)
    =
    \frac{C_1}{2 (1 - t)^d}
    \Hilb_{(-\bs{\alpha},\bs{\alpha});\beta}^{\Sp^1}(t)
    + \frac{C_2}{2 (1 - t^2)^n (1 + t)^d}.
\end{equation}
\end{remark}

In addition, we will need the $\N^2$-graded Hilbert series of the covariants of a cotangent-lifted representation
in the next section. A representation $V = V_{\bs{\alpha},d}\simeq \C^{2n+d}$
is self-dual so that the cotangent lift of $V$ is isomorphic to $2V$, i.e., $V_{2\bs{\alpha},2d}\simeq \C^{4n+2d}$
where $2\bs{\alpha}$ indicates $\bs{\alpha}$ concatenated with itself. As in Section~\ref{subsec:S1Cot}, we use the
grading $(d_1,d_2)$ induced by the decomposition $2V$ with formal variables $(s,t)$. Adapting Theorem~\ref{thrm:O2Max}
to the representation $V_{2\bs{\alpha},2d}\simeq \C^{4n+2d}$, applying the substitutions $t_i = s$ for $1\leq i \leq n$
and $t_i = t$ for $n < i \leq 2n$ to compute as in the proof of Corollary~\ref{cor:S1CotanBivar}, and using analytic
continuation \cite[Section~3.3]{HerbigHerdenSeatonT2} when the $\alpha_i$ are not distinct, we obtain the following.

\begin{corollary}[Bivariate Hilbert series of cotangent-lifted $\OO_2$-covariants]
\label{cor:O2Bivar}
Let $\bs{\alpha} = (\alpha_1,\ldots,\alpha_n)$ where $n\geq 1$, each $\alpha_i > 0$, and $d\geq 0$, let
$V = V_{\bs{\alpha},d}\simeq \C^{2n+d}$ be the corresponding faithful representation of $\OO_2(\R)$,
and let $W$ be an irreducible representation of $\OO_2(\R)$. The $\N^2$-graded
Hilbert series $\Hilb_{\bs{\alpha},d;W}^{\OO_2,\off}(s,t)$ associated to the cotangent lift of $V$ is given by
\begin{multline}
\label{eq:O2CotanW1dim}
    \lim\limits_{(c_1,\ldots,c_{2n})\to 2\bs{\alpha}}
    \frac{C_1}{2 (1-s)^d(1-t)^d(1 - st)}
     \sum_{\substack{i=1 \\ \zeta^{\alpha_i} = 1}}^n \Bigg(
        \frac{\zeta^\beta s^{\beta/\alpha_i} (1 - s^{-c_{n+i}/c_i}t)^{-1}}
     {\alpha_i (1 - s^2)\prod\limits_{\substack{j=1 \\ j\neq i}}^n F_1(\zeta,i,j)}
    + \frac{\zeta^\beta t^{\beta/\alpha_i}(1 - st^{-c_{i}/c_{n+i}})^{-1}}{ \alpha_i (1 -  t^2)
        \prod\limits_{\substack{j=1 \\ j\neq i}}^n F_2(\zeta,i,j)}
     \Bigg)\\
    + \frac{C_2}{2 (1 - s^2)^n (1 - t^2)^n (1 + s)^d (1 + t)^d},
\end{multline}
where $\beta$, $C_1$, and $C_2$ are as in Theorem~\ref{thrm:O2Max},
\[
    F_1(\zeta,i,j)
    =   (1 - \zeta^{-\alpha_j} s^{(c_i - c_j)/c_i})
        (1 - \zeta^{\alpha_j} s^{(\alpha_i + \alpha_j)/\alpha_i} )
        (1 - \zeta^{-\alpha_{j}} s^{- c_{n+j}/c_i}t)
        (1 - \zeta^{\alpha_{j}} s^{\alpha_{j}/\alpha_i}t)
\]
and
\[
    F_2(\zeta,i,j)
    =   (1 - \zeta^{-\alpha_j} st^{- c_j/c_{n+i}})
        (1 - \zeta^{\alpha_j} st^{\alpha_j/\alpha_{i}} )
        (1 - \zeta^{-\alpha_{j}} t^{(c_{n+i} - c_{n+j})/c_{n+i}})
        (1 - \zeta^{\alpha_{j}} t^{(\alpha_{i} + \alpha_{j})/\alpha_{i}} ).
\]
If the $\alpha_i$ are distinct, then the limit is unnecessary and we may substitute
$c_i = \alpha_i$ and $c_{n+i} = \alpha_i$ for $1\leq i\leq n$.
\end{corollary}

% xxxxxxxxxxxxxxxxxxxxxxxxxxxxxxxxxxxxxxxxxxxxxxxxxxxxxxxxxxxxxxxxxxxxxxxxx

\subsection{Symplectic quotients by $\OO_2$}
\label{subsec:O2Symplectic}

Let $V = V_{\bs{\alpha},d}\simeq \C^{2n+d}$ be a representation of $\OO_2(\R)$ as in the previous section.
We consider coordinates $(\bs{x},\bs{y}):=(x_{1,1},x_{1,2},x_{2,1},\ldots,x_{n,2},y_1,\ldots,y_d)$
for $V$ where $(x_{i,1},x_{i,2})$ are coordinates for $\tau_{\alpha_i}$ and the $y_i$
are coordinates for each copy of $\det$. As the moment map $J\co V\to\mathfrak{g}^\ast$
depends only on the connected component of the identity of the group, it coincides with the moment
map of the underlying circle action,
\[
    J(\bs{x},\bs{y})
    =
    \sum\limits_{i=1}^n \alpha_i \big(x_{i,1} \overline{x_{i,1}} - x_{i,2} \overline{x_{i,2}}\big).
\]
It follows from this expression that $J\big(r_\theta(\bs{x},\bs{y})\big) = J(\bs{x},\bs{y})$ for any
$r_\theta\in\SO_2(\R)$, i.e., $J$ is invariant under the action of $\SO_2(\R)\leq\OO_2(\R)$. However,
if $s_\theta\in\OO_2(\R)\smallsetminus\SO_2(\R)$, then $s_\theta$
can be expressed as an element of $\SO_2(\R)$ followed by swapping each $x_{i,1}$ with $x_{i,2}$, and therefore
we have $J\big(s_\theta(\bs{x},\bs{y})\big) = - J(\bs{x},\bs{y})$. Hence, the moment map is never an $\OO_2(\R)$-invariant,
and the relationship between the on- and off-shell invariants is not as simple as in the case of the circle.

However, the Hilbert series of the on-shell invariants of the symplectic quotient can be computed using the techniques
of \cite[Section~6.3]{HerbigSchwarzSeaton2} as applied to the case $\SU_2$ in \cite[Proposition~2.1]{HerbigHerdenSeatonSU2}.
Note that the shell only depends on the action of the Lie algebra of $\OO_2(\R)$ and hence on the corresponding action
of $\SO_2(\R)$. Therefore, if $V = d\det$, then the shell is trivially equal to $V$. Otherwise, $V$
has at least one $\tau_\alpha$ summand with $\alpha > 0$. As a representation of $\SO_2(\R) \simeq \Sp^1$, $\tau_\alpha$ has weight vector
$(-\alpha,\alpha)$ which by \cite[Theorem~3.2]{HerbigSchwarz} is \emph{$1$-large}; see that reference for the definition.
It follows that $V$ is $1$-large by \cite[Theorem~3.1]{HerbigSchwarz} so that the complexification of the shell is a
reduced, irreducible complete intersection.

Let $\mu = J\otimes_\R\C$ denote the complexification of $J$.
The Koszul complex of $\mu$ is a free resolution of $\C[V\oplus V^\ast]/(\mu)$ yielding an exact sequence
\[
    0   \to \C[V\oplus V^\ast]\otimes \mathfrak{o}_2 \to \C[V\oplus V^\ast] \to \C[V\oplus V^\ast]/(\mu)    \to 0.
\]
Note that $\mathfrak{o}_2\simeq\det$ as an $\OO_2(\C)$-module and the elements of $\mathfrak{o}_2$ are in degree
$(1,1)$ with respect to the bigrading. Taking invariants and using the fact that
$(\C[V\oplus V^\ast]\otimes \det)^{\OO_2(\C)}$ is the module of covariants with $W = \det$, it follows that
the Hilbert series $\Hilb_{\bs{\alpha},d}^{\OO_2,\on}(s,t)$ of the on-shell invariants of the symplectic quotient
associated to $V$ is
\[
    \Hilb_{\bs{\alpha},d}^{\OO_2,\on}(s,t)
    =
    \Hilb_{\bs{\alpha},d}^{\OO_2,\off}(s,t) - st \Hilb_{\bs{\alpha},d;\det}^{\OO_2,\off}(s,t).
\]
Combining this with Corollary~\ref{cor:O2Bivar} yields the following.

\begin{corollary}[Bivariate Hilbert series of on-shell invariants of an $\OO_2$-symplectic quotient]
\label{cor:O2SymplecticBivar}
Let $\bs{\alpha} = (\alpha_1,\ldots,\alpha_n)$ where $n\geq 1$, each $\alpha_i > 0$, and $d\geq 0$, and let
$V = V_{\bs{\alpha},d}\simeq \C^{2n+d}$ be the corresponding faithful representation of $\OO_2(\R)$. Then the bivariate
Hilbert series $\Hilb_{\bs{\alpha},d}^{\OO_2,\on}(s,t)$ of the real on-shell invariants of the corresponding
symplectic quotient is given by
\begin{multline}
\label{eq:O2SymplecticBivar}
    \lim\limits_{(c_1,\ldots,c_{2n})\to 2\bs{\alpha}}
    \frac{1}{2 (1-s)^d(1-t)^d}
     \sum_{\substack{i=1 \\ \zeta^{\alpha_i} = 1}}^n \Bigg(\,
        \frac{(1 - s^{-c_{n+i}/c_i}t)^{-1}}
     {\alpha_i (1 - s^2)\prod\limits_{\substack{j=1 \\ j\neq i}}^n F_1(\zeta,i,j)}
    +
    \frac{(1 - st^{-c_{i}/c_{n+i}})^{-1}}{ \alpha_i (1 -  t^2)
        \prod\limits_{\substack{j=1 \\ j\neq i}}^n F_2(\zeta,i,j)}
     \Bigg)
    \\
    + \frac{1 + st}{2 (1 - s^2)^n (1 - t^2)^n (1 + s)^d (1 + t)^d},
\end{multline}
where $F_1(\zeta,i,j)$ and $F_2(\zeta,i,j)$ are as defined in Corollary~\ref{cor:O2Bivar}.

Adopting the convention that $\alpha_{n+i}:= \alpha_i$ for $1\leq i\leq n$, the univariate
Hilbert series $\Hilb_{\bs{\alpha},d}^{\OO_2,\on}(t)$ of the real on-shell invariants of the
symplectic quotient is given by
\begin{equation}
\label{eq:O2SymplecticUnivar}
    \lim\limits_{(c_1,\ldots,c_{2n})\to 2\bs{\alpha}}
    \frac{1}{2 (1-t)^{2d}}
     \sum_{\substack{i=1 \\ \zeta^{\alpha_i} = 1}}^{2n}
        \frac{1}
     {\alpha_i \prod\limits_{\substack{j=1 \\ j\neq i}}^{2n}
        (1 - \zeta^{-\alpha_j} t^{(c_i - c_j)/c_i})
        (1 - \zeta^{\alpha_j} t^{(\alpha_i + \alpha_j)/\alpha_i} )}
    + \frac{1 + t^2}{2 (1 - t^2)^{2n}(1 + t)^{2d}}.
\end{equation}
\end{corollary}

The univariate Hilbert series of on-shell invariants can be obtained from the bivariate Hilbert series by
taking the limit as $s\to t$. In particular, the singularities at $s = t$ are removable.

\begin{remark}
\label{rem:O2SymplecticAsS1}
It will be helpful in Section~\ref{subsec:O2Laurent} to observe the following. Using Equation~\eqref{eq:O2MaxW1dimAsS1}
in Remark~\ref{rem:O2MaxAsS1}, Equation~\eqref{eq:O2SymplecticBivar} can also be written
\begin{equation}
\label{eq:O2CotanAsS1}
    \Hilb_{\bs{\alpha},d}^{\OO_2,\on}(s,t) =
    \frac{1}{2(1 - s)^d(1 - t)^d}\Hilb_{(-\bs{\alpha}, \bs{\alpha})}^{\Sp^1,\on}(s,t) + \frac{1 + st}{2 (1 - s^2)^n (1 - t^2)^n (1 + s)^d (1 + t)^d},
\end{equation}
where $\Hilb_{(-\bs{\alpha}, \bs{\alpha})}^{\Sp^1,\on}(s,t)$ is the Hilbert series of the on-shell invariants of the symplectic quotient
with weight vector $(-\bs{\alpha}, \bs{\alpha})$ as in Corollary~\ref{cor:S1SymplecticBivar}. Similarly,
Equation~\eqref{eq:O2SymplecticUnivar} can be written
\begin{equation}
\label{eq:O2SymplecticUnivarAsS1}
    \Hilb_{\bs{\alpha},d}^{\OO_2,\on}(t) =
    \frac{1}{2 (1-t)^{2d}} \Hilb_{(-\bs{\alpha}, \bs{\alpha})}^{\Sp^1,\on}(t)
    + \frac{1 + t^2}{2 (1 - t^2)^{2n}(1 + t)^{2d}}.
\end{equation}
\end{remark}

% xxxxxxxxxxxxxxxxxxxxxxxxxxxxxxxxxxxxxxxxxxxxxxxxxxxxxxxxxxxxxxxxxxxxxxxxx

\subsection{The Laurent coefficients for symplectic quotients by $\OO_2(\R)$}
\label{subsec:O2Laurent}

Let $V = V_{\bs{\alpha},d}\simeq \C^{2n+d}$ and $W$ be representations of $\OO_2(\R)$ with $W$ irreducible
as above. As in Section~\ref{subsec:S1Laurent}, we consider the Laurent coefficients $\gamma_m^{\OO_2}(\bs{\alpha},d;W)$
of $\Hilb_{\bs{\alpha},d;W}^{\OO_2}(t)$ at $t = 1$ where $\gamma_0^{\OO_2}(\bs{\alpha},d;W)$ occurs in degree $1 - 2n - d$.
If $V$ contains no $\tau_{\alpha_i}$ summands, i.e., $\bs{\alpha}$ is empty, then $V = d\det$.
Following the proof of Theorem~\ref{thrm:O2Max}, Equations~\eqref{eq:O2MaxIntegral1} and \eqref{eq:O2MaxIntegral2}, the Hilbert series is simply
\[
    \Hilb_{\bs{\alpha},d;W}^{\OO_2}(t) = \frac{2 - C_1}{2 (1 - t)^d}
    +
    \frac{C_2}{2(1 + t)^d}
\]
with $C_1$ and $C_2$ as defined in Theorem~\ref{thrm:O2Max}, and the Laurent expansion is given by
\[
    \Hilb_{\bs{\alpha},d;W}^{\OO_2}(t)
    =
    \frac{2 - C_1}{2 (1 - t)^d} +
    \sum\limits_{m=0}^\infty \frac{C_2}{2^{d+m}} {d + m - 1 \choose m} (1 - t)^m.
\]
Hence, we will hereafter ignore this case and assume that $V$ contains at least one $\tau_{\alpha_i}$
summand, in which case we can assume by Remark~\ref{rem:O2MaxAsS1} that $V$ is faithful.

Using Equation~\eqref{eq:O2UnivarW1dimAsS1} and the fact that the second term has pole order $n$,
only the first term contributes to the first two Laurent coefficients of $\Hilb_{\bs{\alpha},d;W}^{\OO_2}(t)$
except for small values of $n$. Considering these values by inspection, we have the following.

\begin{corollary}[$\gamma_0$ and $\gamma_1$ for $\OO_2$-covariants]
\label{cor:O2Laurent}
Let $V = V_{\bs{\alpha},d}\simeq \C^{2n+d}$ and $W$ be representations of $\OO_2(\R)$ with $V$ faithful,
$n\geq 1$, and $W$ irreducible. Let $\beta = 0$ if $W$ is $\det$ or the trivial representation and otherwise
let $W = \tau_\beta$. Then
\begin{align*}
    \gamma_0^{\OO_2}(\bs{\alpha},d;W)
    &=      \begin{cases}
                \frac{C_1 \gamma_0^{\Sp^1}((-\bs{\alpha},\bs{\alpha});\beta)}{2} + \frac{C_2}{4},
                    &       \text{if } n = 1 \text{ and } d = 0,
                \\
                \frac{C_1 \gamma_0^{\Sp^1}((-\bs{\alpha},\bs{\alpha});\beta)}{2},
                    &       \text{otherwise,}
            \end{cases}
    \\
    \gamma_1^{\OO_2}(\bs{\alpha},d;W)
    &=      \begin{cases}
                \frac{C_1 \gamma_1^{\Sp^1}((-\bs{\alpha},\bs{\alpha});\beta)}{2} + \frac{C_2}{8},
                    &       \text{if } n + d \leq 2,
                \\
                \frac{C_1 \gamma_1^{\Sp^1}((-\bs{\alpha},\bs{\alpha});\beta)}{2},
                    &       \text{otherwise,}
            \end{cases}
\end{align*}
where $C_1$ and $C_2$ are as defined in Theorem~\ref{thrm:O2Max}.
\end{corollary}

In the same way, the following is a consequence of Equation~\eqref{eq:O2SymplecticUnivarAsS1} and inspection of low-dimensional cases.

\begin{corollary}[$\gamma_0$, $\gamma_1$, $\gamma_2$, and $\gamma_3$ for on-shell invariants of $\OO_2$-symplectic quotients]
\label{cor:O2LaurentSymplectic}
Let $V = V_{\bs{\alpha},d}\simeq \C^{2n+d}$ be a faithful representation of $\OO_2(\R)$ with $n\geq 1$. Then
\begin{align*}
    \gamma_0^{\OO_2,\on}(\bs{\alpha},d)
    &=      \begin{cases}
                \frac{ \gamma_0^{\Sp^1,\on}((-\bs{\alpha},\bs{\alpha}))}{2} + \frac{1}{4},
                    &       \text{if } n = 1 \text{ and } d = 0,
                \\
                \frac{ \gamma_0^{\Sp^1,\on}((-\bs{\alpha},\bs{\alpha}))}{2},
                    &       \text{otherwise,}
            \end{cases}
    \\
    \gamma_1^{\OO_2,\on}(\bs{\alpha},d)
    &=      0, \text{ and }
    \\
    \gamma_2^{\OO_2,\on}(\bs{\alpha},d) = \gamma_3^{\OO_2,\on}(\bs{\alpha},d)
    &=      \begin{cases}
                \frac{\gamma_2^{\Sp^1,\on}((-\bs{\alpha},\bs{\alpha}))}{2} + \frac{1}{16},
                    &       \text{if } n + d \leq 2,
                \\
                \frac{\gamma_2^{\Sp^1,\on}((-\bs{\alpha},\bs{\alpha}))}{2},
                    &       \text{otherwise.}
            \end{cases}
\end{align*}
\end{corollary}

For $m \leq 3$, the coefficients $\gamma_m^{\Sp^1,\on}((-\bs{\alpha},\bs{\alpha}))$ for $\Sp^1$-symplectic quotients were computed
in \cite[Theorem~5.1]{HerbigSchwarzSeaton}.

% xxxxxxxxxxxxxxxxxxxxxxxxxxxxxxxxxxxxxxxxxxxxxxxxxxxxxxxxxxxxxxxxxxxxxxxxx

\section{Other semidirect products of $\Sp^1$ by finite groups}
\label{sec:O2OtherSemidirect}

The approach of Section~\ref{sec:O2} can be applied to other extensions of the circle by a finite group.
As an example, consider the extension $\Sp^1\rtimes\Z/4\Z$ where the elements of $\Z/4\Z$ of order $4$
act on $\Sp^1$ via $t\mapsto t^{-1}$ and the other elements act trivially. That is, if $\gamma$ is a
generator of $\Z/4\Z$, then the multiplication of $\Sp^1\rtimes\Z/4\Z$ is given by
\[
    (z_1, \gamma^j)(z_2, \gamma^k)
    =
    (z_1 z_2^{(-1)^j}, \gamma^{j+k})
\]
for $z_1, z_2\in\Sp^1$. The representation
of $\Sp^1\rtimes\Z/4\Z$ induced by the representation $\epsilon_a$ of the normal subgroup $\Sp^1$ is denoted
$\nu_a\co\Sp^1\rtimes\Z/4\Z\to\U_4$ and given by
\begin{align*}
        (z, 1)
        &\mapsto
        \begin{pmatrix} z^a & 0 & 0 & 0 \\ 0 & z^{-a} & 0 & 0 \\ 0 & 0 & z^a & 0 \\ 0 & 0 & 0 & z^{-a} \end{pmatrix},
    &
        (z, \gamma)
        &\mapsto
        \begin{pmatrix} 0 & z^a & 0 & 0  \\ 0 & 0 & z^{-a} & 0 \\ 0 & 0 & 0 & z^a \\ z^{-a} & 0 & 0 & 0 \end{pmatrix},
    \\
        (z, \gamma^2)
        &\mapsto
        \begin{pmatrix} 0 & 0 & z^a & 0  \\ 0 & 0 & 0 & z^{-a} \\ z^a & 0 & 0 & 0 \\ 0 & z^{-a} & 0 & 0 \end{pmatrix},
    &
        (z, \gamma^3)
        &\mapsto
        \begin{pmatrix} 0 & 0 & 0 & z^a  \\ z^{-a} & 0 & 0 & 0 \\ 0 & z^a & 0 & 0 \\ 0 & 0 & z^{-a} & 0 \end{pmatrix},
\end{align*}
where $z\in\Sp^1$.

For simplicity, we consider a representation of the form $V = \bigoplus\limits_{i=1}^n \nu_{a_i}$ where each $a_i > 0$.
The maximally graded Hilbert series of the invariants can be computed using the same methods as in Theorem~\ref{thrm:O2Max}.
Specifically, for the connected components corresponding to elements of the form $(z, \gamma)$ and $(z, \gamma^3)$, each
integral in the Molien-Weyl Theorem is simply
\[
    \frac{1}{8\pi\sqrt{-1}} \int\limits_{\Sp^1} \frac{dz}{z \prod\limits_{i=1}^n (1 - t_i^4)}
    =
    \frac{1}{4\prod\limits_{i=1}^n (1 - t_i^4)}.
\]
For the connected component associated to elements of the form $(z, 1)$, the integral is
\[
    \frac{1}{8\pi\sqrt{-1}} \int\limits_{\Sp^1} \frac{dz}{z \prod\limits_{i=1}^n (1 - t_i z^{a_i})^2(1 - t_i z^{-a_i})^2}.
\]
Choosing an $i$ and an $a_i$th root of unity $\zeta_0$ and rewriting the integrand as
\[
    \frac{z^{2a_i-1}}{(1 - t_i z^{a_i})^2
        (z - \zeta_0 t_i^{1/a_i})^2
        \prod\limits_{\substack{\zeta^{a_i} = 1\\ \zeta\neq\zeta_0}}(z - \zeta t_i^{1/a_i})^2
        \prod\limits_{\substack{j=1\\j\neq i}}^n (1 - t_j z^{a_j})^2(1 - t_j z^{-a_j})^2},
\]
the residue at $z = \zeta_0 t_i^{1/a_i}$ is given by
\[
    \frac{\partial}{\partial z} \left.\left( \frac{z^{2a_i-1}}{(1 - t_i z^{a_i})^2
        \prod\limits_{\substack{\zeta^{a_i} = 1\\ \zeta\neq\zeta_0}}(z - \zeta t_i^{1/a_i})^2
        \prod\limits_{\substack{j=1\\j\neq i}}^n (1 - t_j z^{a_j})^2(1 - t_j z^{-a_j})^2}\right)\right\rvert_{z = \zeta_0 t_i^{1/a_i}}.
\]

For the connected component associated to elements of the form $(z, \gamma^2)$, the integral is
\[
    \frac{1}{8\pi\sqrt{-1}} \int\limits_{\Sp^1} \frac{dz}
            {z \prod\limits_{i=1}^n (1 - t_i z^{-a_i})(1 + t_i z^{-a_i})(1 - t_i z^{a_i})(1 + t_i z^{a_i})}.
\]
Rewriting the integrand as
\[
    \frac{z^{2a_i-1}}{
        (z^{a_i} - t_i) (z^{a_i} + t_i)
        (1 - t_i z^{a_i})(1 + t_i z^{a_i})
        \prod\limits_{\substack{j=1\\ j\neq i}}^n
            (1 - t_j z^{-a_j})
            (1 + t_j z^{-a_j})
            (1 - t_j z^{a_j})
            (1 + t_j z^{a_j})}
\]
and choosing an $a_i$th root of unity $\zeta_0$,
the residue at a pole of the form $z = \zeta_0 t_i^{1/a_i}$ is given by
\[
    \frac{1 }{
        2a_i (1 - t_i^4)
        \prod\limits_{\substack{j=1\\ j\neq i}}^n
            (1 - \zeta_0^{-a_j} t_i^{-a_j/a_i} t_j )
            (1 + \zeta_0^{-a_j} t_i^{-a_j/a_i} t_j )
            (1 - \zeta_0^{a_j} t_i^{a_j/a_i} t_j )
            (1 + \zeta_0^{a_j} t_i^{a_j/a_i} t_j )},
\]
and the residue at a pole of the form $z = \zeta_0 (-t_i)^{1/a_i}$ is given by
\[
    \frac{1}{
        2 a_i (1 - t_i^4)
        \prod\limits_{\substack{j=1\\ j\neq i}}^n
            (1 - \zeta_0^{-a_j} (-t_i)^{-a_j/a_i} t_j )
            (1 + \zeta_0^{-a_j} (-t_i)^{-a_j/a_i} t_j )
            (1 - \zeta_0^{a_j} (-t_i)^{a_j/a_i} t_j )
            (1 + \zeta_0^{a_j} (-t_i)^{a_j/a_i} t_j )}.
\]
Hence the maximally graded Hilbert series of the invariants of the representation $V = \bigoplus\limits_{i=1}^n \nu_{a_i}$
is given by
\begin{align*}
    &\frac{1}{2\prod\limits_{i=1}^n (1 - t_i^4)}
    +
    \frac{1}{4} \sum\limits_{i=1}^n  \frac{1}{a_i^4 (1 - t_i^2)^4} \sum\limits_{\zeta^{a_i} = 1}
        \frac{ (2a_i - 1)  D_i
                - (\zeta t_i^{1/a_i})^{3-2a_i} D_i^\prime}
            {
                \prod\limits_{\substack{j=1\\j\neq i}}^n
                    (1 - \zeta^{a_j}  t_i^{a_j/a_i} t_j )^4
                    (1 - \zeta^{-a_j} t_i^{-a_j/a_i} t_j)^4}
    \\&\quad +
    \sum\limits_{i=1}^n \frac{1}{8a_i (1 - t_i^4)}
    \sum\limits_{\zeta^{a_i} = 1} \Bigg[
        \frac{1}{\prod\limits_{\substack{j=1\\ j\neq i}}^n
            (1 - \zeta^{-a_j} t_i^{-a_j/a_i} t_j )
            (1 + \zeta^{-a_j} t_i^{-a_j/a_i} t_j )
            (1 - \zeta^{a_j} t_i^{a_j/a_i} t_j )
            (1 + \zeta^{a_j} t_i^{a_j/a_i} t_j )}
    \\&+ \quad
        \frac{1}{\prod\limits_{\substack{j=1\\ j\neq i}}^n
            (1 - \zeta^{-a_j} (-t_i)^{-a_j/a_i} t_j )
            (1 + \zeta^{-a_j} (-t_i)^{-a_j/a_i} t_j )
            (1 - \zeta^{a_j} (-t_i)^{a_j/a_i} t_j )
            (1 + \zeta^{a_j} (-t_i)^{a_j/a_i} t_j )}
    \Bigg],
\end{align*}
where
\[
    D_i = a_i^2 (1 - t_i^2)^2
        \prod\limits_{\substack{j=1\\j\neq i}}^n
            (1 - \zeta^{a_j} t_i^{a_j/a_i} t_j )^2
            (1 - \zeta^{-a_j} t_i^{-a_j/a_i} t_j )^2,
\]
and
\[
    D_i^\prime = \frac{d}{dz}\left.\left(
    (1 - t_i z^{a_i})^2
        \prod\limits_{\substack{\eta^{a_i} = 1\\ \eta\neq\zeta}}(z - \eta t_i^{1/a_i})^2
        \prod\limits_{\substack{j=1\\j\neq i}}^n (1 - t_j z^{a_j})^2(1 - t_j z^{-a_j})^2
    \right)\right\rvert_{z = \zeta t_i^{1/a_i}}.
\]
One could also compute the first few Laurent coefficients of the corresponding univariate Hilbert series using the
techniques of Section~\ref{subsec:S1Laurent} and \ref{subsec:O2Laurent}, though it is clear that the level of
complexity increases considerably.

% xxxxxxxxxxxxxxxxxxxxxxxxxxxxxxxxxxxxxxxxxxxxxxxxxxxxxxxxxxxxxxxxxxxxxxxxx
% xxxxxxxxxxxxxxxxxxxxxxxxxxxxxxxxxxxxxxxxxxxxxxxxxxxxxxxxxxxxxxxxxxxxxxxxx
% xxxxxxxxxxxxxxxxxxxxxxxxxxxxxxxxxxxxxxxxxxxxxxxxxxxxxxxxxxxxxxxxxxxxxxxxx

\section{Algorithms to compute the Hilbert series}
\label{sec:Algorithms}

The formulas given in Theorems~\ref{thrm:S1Max} and \ref{thrm:O2Max} and Corollaries~\ref{cor:S1Univar},
\ref{cor:S1CotanBivar}, \ref{cor:S1CotanUnivar}, \ref{cor:S1SymplecticBivar},
\ref{cor:O2Univar}, \ref{cor:O2Bivar}, and \ref{cor:O2SymplecticBivar} each indicate an algorithm for the
computation of the corresponding Hilbert series very similar to the algorithms described in
\cite[Section~4]{HerbigSeaton} and \cite[Section~4]{CowieHerbigSeatonHerden}; see also
\cite[Section~6]{CayresPintoHerbigHerdenSeaton} and \cite[Section~3.3]{HerbigHerdenSeatonT2}.
We give a brief description of this algorithm for the case of
Theorem~\ref{thrm:S1Max}, as the others are the same with minor modifications, and refer the reader to the above
references for more details.

Let $a\in \Z$, $a>0$. For a formal power series $F(\bs{t})$ in the variables $\bs{t} = (t_1,\ldots,t_n)$, let $U_{a,i}$ denote the operator
\[
    U_{a,i} F(\bs{t}) = \frac{1}{a} \sum\limits_{\zeta^a = 1} F(t_1,\ldots,\zeta t_i^{1/a},\ldots, t_n).
\]
Let $\bs{t}_i = (t_1,\ldots,t_{i-1},t_{i+1},\ldots,t_n)$, and note that if
$F(\bs{t}) = \sum_{d=0}^\infty c_d(\bs{t}_i) t_i^d$, then
\begin{equation}
\label{eq:AlgU}
    U_{a,i} F(\bs{t})   =   \sum\limits_{d = 0}^\infty c_{ad}(\bs{t}_i) t_i^d.
\end{equation}

Using this operator, the first sum in Equation~\eqref{eq:S1Max} can be rewritten as $\sum_{i=1}^k U_{\alpha_i,i} \Phi_i(\bs{t})$ where
\[
    \Phi_i(\bs{t}) = \frac{t_i^b}{
        \prod\limits_{\substack{j=1 \\ j\neq i}}^n (1 - t_j t_i^{a_j})}.
\]
The function $\Phi_i(\bs{t})$ can be written as $P(\bs{t})/Q(\bs{t})$ where $P(\bs{t})$ is a Laurent monomial in $\bs{t}$ and $Q(\bs{t})$ is a product of
factors of the form $(1 - t_j^p t_i^q)$ with $q > 0$. Then $U_{\alpha_i,i}(P(\bs{t})/Q(\bs{t}))$ is a rational function with denominator
given by the replacement rule
\[
    (1 - t_j^p t_i^{q} )   \mapsto (1 - t_j^{\alpha_ip/\gcd(\alpha_i,q)} t_i^{q/\gcd(\alpha_i,q)} )^{\gcd(\alpha_i,q)}
\]
applied to $Q(\bs{t})$.
The degrees of the numerator and denominator of $U_{\alpha_i,i}(P(\bs{t})/Q(\bs{t}))$ can be computed using
\cite[Equation~(16)]{HerbigHerdenSeatonT2}, so that knowing the denominator and the degree of the numerator,
the numerator of $U_{\alpha_i,i}(P(\bs{t})/Q(\bs{t}))$ can be computed using Equation~\eqref{eq:AlgU}. This yields a computation of the
first sum in Equation~\eqref{eq:S1Max}, and the second sum can be computed directly by searching through the finite
set of possible elements of $\mathcal{S}_{\bs{a},b}$ as defined in Equation~\eqref{eq:S1MaxS}, e.g., by using the
\texttt{FrobeniusSolve} function on \emph{Mathematica} \cite{Mathematica}. However, as noted in
Remark~\ref{rem:S1MaxSEmpty}, one can always reduce to a case where $\mathcal{S}_{\bs{a},b}$ is empty so that the
second sum vanishes.

These algorithms have been implemented on \emph{Mathematica} and are available from the authors upon request.

% xxxxxxxxxxxxxxxxxxxxxxxxxxxxxxxxxxxxxxxxxxxxxxxxxxxxxxxxxxxxxxxxxxxxxxxxx
% xxxxxxxxxxxxxxxxxxxxxxxxxxxxxxxxxxxxxxxxxxxxxxxxxxxxxxxxxxxxxxxxxxxxxxxxx
% xxxxxxxxxxxxxxxxxxxxxxxxxxxxxxxxxxxxxxxxxxxxxxxxxxxxxxxxxxxxxxxxxxxxxxxxx

\bibliographystyle{amsplain}
\bibliography{BHHKSW}

\end{document}